\documentclass[12pt,reqno]{amsart}
\usepackage{amsmath,amssymb,amsfonts,amstext,amsthm}
\usepackage{url}
\usepackage{graphicx}
\input xy 
\xyoption{all}
\usepackage[all]{xy}
\usepackage{mathrsfs}

\graphicspath{{images/}{../images/}}
\usepackage{subfiles}
\usepackage{tikz}

\newcommand{\nwc}{\newcommand}

\nwc{\aaa}{\mathcal{F}}
\nwc{\aap}{\mathcal{F}_{P}}
\nwc{\al}{\alpha}

\nwc{\C}{\mathbb{C}}
\nwc{\cb}{\overline{C}}
\nwc{\ccc}{\mathfrak{c}}
\nwc{\ch}{\widehat{C}}
\nwc{\cin}{\textbf{(v)}}
\nwc{\cl}{C'}
\nwc{\cp}{\mathcal{C}_{P}}
\nwc{\cpll}{\mathfrak{c}_{P'}}
\nwc{\ct}{\widetilde{C}}

\nwc{\dd}{\mathcal{L}}
\nwc{\ddd}{\mathfrak{d}}
\nwc{\ddl}{\mathcal{L}'}
\nwc{\dlp}{\delta_{P}}
\nwc{\doi}{\textbf{(ii)}}

\nwc{\enq}{$$}

\nwc{\fl}{\flushleft}
\nwc{\fff}{\mathcal{F}}
\nwc{\ffp}{\mathcal{F}_{P}}
\nwc{\ffq}{\mathcal{F}_{Q}}
\nwc{\ffl}{\mathcal{F}'}

\nwc{\G}{\mathcal{G}}
\nwc{\Ga}{\Gamma}
\nwc{\gtl}{\widetilde{g}}

\nwc{\hra}{\hookrightarrow}
\nwc{\hua}{h^{1}(C,\aaa )}

\nwc{\kk}{{\rm K}}

\nwc{\llb}{\mathcal{L}}

\nwc{\mb}{\mathbb}
\nwc{\mc}{\mathcal}
\nwc{\mm}{\mathfrak{m}}
\nwc{\mmp}{\mathfrak{m}_{P}}
\nwc{\mpd}{\mathfrak{m}_{P}^{2}}

\nwc{\nn}{\mathbb{N}}

\nwc{\ob}{\overline{\mathcal{O}}}
\nwc{\obr}{\mathcal{O}^*}
\nwc{\obp}{\overline{\mathcal{O}}_P}
\nwc{\och}{\mathcal{O}_{\widehat{C}}}
\nwc{\oh}{\widehat{\mathcal{O}}}
\nwc{\ohp}{\widehat{\mathcal{O}}_{P}}
\nwc{\ol}{\mathcal{O}'}
\nwc{\oma}{\Omega (\mathfrak{a})}
\nwc{\omo}{\Omega (\mathcal{O})}
\nwc{\oo}{\mathcal{O}}
\nwc{\op}{\mathcal{O}_P}
\nwc{\opc}{\mathcal{O}_{P,C}}
\nwc{\oph}{\widehat{\mathcal{O}}_{P}}
\nwc{\opl}{\mathcal{O}_{P}'}
\nwc{\oplc}{\mathcal{O}_{P,C}'}
\nwc{\opll}{\mathcal{O}_{P'}}
\nwc{\opt}{\widetilde{\mathcal{O}}_{P}}
\nwc{\optt}{{\mathcal{O}}_{\widetilde{P}}}
\nwc{\oq}{\mathcal{O}_{Q}}
\nwc{\oqt}{\widetilde{\mathcal{O}}_{Q}}
\nwc{\ot}{\widetilde{\mathcal{O}}}
\nwc{\overop}{\overline{\oo}_{P}}

\nwc{\pb}{\overline{P}}
\nwc{\pbb}{P^*}
\nwc{\pbi}{\overline{P_{i}}}
\nwc{\pbr}{\overline{P_{r}}}
\nwc{\pgmd}{\mathbb{P}^{g+2}}
\nwc{\pgmu}{\mathbb{P}^{g+1}}
\nwc{\ph}{\widehat{P}}
\nwc{\pp}{\mathbb{P}}
\nwc{\pt}{\widetilde{P}}
\nwc{\ptl}{\widetilde{P}}
\nwc{\pum}{\mathbb{P}^{1}}

\nwc{\qh}{\widehat{Q}}
\nwc{\qtl}{\widetilde{Q}}
\nwc{\qua}{\textbf{(iv)}}

\nwc{\ra}{\longrightarrow}
\nwc{\rh}{\widehat{R}}

\nwc{\sei}{\textbf{(vi)}}
\nwc{\sep}{\beq\ast\ \ast\ \ast\enq}
\nwc{\sig}{\sigma}
\nwc{\Sig}{\Sigma}
\nwc{\ssp}{S_{P}}
\nwc{\sss}{{\rm S}}

\nwc{\tre}{\textbf{(iii)}}

\nwc{\um}{\textbf{(i)}}

\nwc{\vpb}{v_{\overline{P}}}
\nwc{\vtxp}{\widetilde{V}_{x,P}}
\nwc{\vxp}{V_{x,P}}

\let \wt=\widetilde
\let \mc=\mathcal

\nwc{\wh}{\widehat{\omega}}
\nwc{\whp}{\widehat{\omega}_{P}}
\nwc{\woch}{\omega\cdot\mathcal{O}_{\widehat{C}}}
\nwc{\woh}{\omega\cdot\widehat{\mathcal{O}}}
\nwc{\ww}{\omega}
\nwc{\wwb}{\omega^*}
\nwc{\wwct}{\omega _{\widetilde{C}}}
\nwc{\wwh}{\widehat{\omega}}
\nwc{\wwhp}{\widehat{\omega}_P}
\nwc{\wwp}{\omega _{P}}
\nwc{\wwt}{\widetilde{\omega}}
\nwc{\wwtp}{\widetilde{\omega}_P}

\nwc{\zz}{\mathbb{Z}}

\newtheorem{coro}{Corollary}[section]
\newtheorem{lemma}[coro]{Lemma}
\newtheorem{prop}[coro]{Proposition}

\newtheorem{thm}[coro]{Theorem}

\theoremstyle{definition}
\newtheorem{rem}[coro]{Remark}

\newtheorem{theorem}{Theorem}

\numberwithin{equation}{section}

\let \fl=\flushleft
\let \ep=\epsilon

\let \sub=\subset
\let \be=\beta
\let \al=\alpha
\let \pr=\prime

\let \mf=\mathfrak

\begin{document}

\title{Real inflection points of real hyperelliptic curves}

\author{Indranil Biswas}

\address{School of Mathematics, Tata Institute for Fundamental Research, Homi Bhabha Road, Mumbai 
400005, India}

\email{indranil@math.tifr.res.in}

\author{Ethan Cotterill}

\address{Instituto de Matem\'atica, UFF, Rua M\'ario Santos Braga, S/N,
24020-140 Niter\'oi RJ, Brazil}

\email{cotterill.ethan@gmail.com}

\author{Cristhian Garay L\'opez}

\address{Instituto de Matem\'atica, UFF, Rua M\'ario Santos Braga, S/N,
24020-140 Niter\'oi RJ, Brazil}

\email{cgaray@impa.br}



\thanks{The first author
is supported by a J. C. Bose Fellowship. The second is supported by CNPq grant 309211/2015-8. The third author is supported by a CNPq PDJ fellowship, grant no. 401565/2014-9.}

\begin{abstract}
Given a real hyperelliptic algebraic curve $X$ with non-empty real part and a real effective divisor $\mc{D}$ arising via pullback from $\mathbb{P}^1$ under the hyperelliptic structure map, we study the real inflection points of the associated complete real 
linear series $|\mc{D}|$ on $X$.


To do so we use Viro's patchworking of real plane curves, recast in the
context of some Berkovich spaces studied by M. Jonsson. Our method gives a simpler and more explicit alternative to limit linear series on metrized complexes of curves, as
developed by O. Amini and M. Baker, for curves embedded in toric surfaces.
\end{abstract}

\maketitle

\tableofcontents

\section{Introduction}

{\it Degeneration} has long been a potent tool in the study of linear series on projective algebraic curves. 
In the 1880's Castelnuovo used degenerations to irreducible rational nodal curves; one hundred years later Eisenbud and Harris \cite{EH} introduced a theory of {\it limit linear series} to deal systematically with degenerations to (abstract) reducible curves. Around the same time that Eisenbud and Harris developed their theory, Viro \cite{V} introduced his {\it patchworking} degeneration for hypersurfaces embedded in toric varieties. In Viro's construction the limit is not algebraic in the usual sense, but comes naturally equipped with a sheaf of piecewise-linear regular functions. This limit is combinatorially robust, and for small values of the deformation parameter the algebraic hypersurfaces retain (some of) its properties. Viro's construction plays a key r\^ole in tropical geometry and in real algebraic geometry, and notably has been used successfully to study linear series over the real numbers. 

\medskip
The third author in his
thesis \cite{G} has used Viro's patchworking construction to study real inflection 
points of real canonical curves of genus four in $\mb{P}^3$. There, the limiting object was equal 
to the dual graph of a necklace of elliptic curves; by combining the real inflectionary behavior 
of a single elliptic curve with Viro's patchworking construction he was able to exhibit canonical 
curves of genus four in $\mb{RP}^3$ with 30 real inflection points. In doing so, he built on 
earlier work of Brugall\'e and L\'opez de Medrano \cite{BL} who used Viro's patchworking to 
systematically construct real algebraic plane curves with the maximal possible number of real 
inflection points. 

\medskip
In this paper, we will apply an enhanced version of patchworking towards the construction of 
real maximally-inflected linear series on {\it hyperelliptic} curves of arbitrary genus. More 
precisely, by combining patchworking with a Berkovich-analytic construction of Jonsson \cite{J} 
we exhibit a {\it metrized complex} of curves, in the sense of \cite{AB}, as the {\it analytic} 
limit of a family of real hyperelliptic curves. As systems of linear series along the smooth curve 
components of a metrized complex, the corresponding limits of linear series arising from our 
construction are similar to those of \cite{AB}. Our construction is, however, more explicit, as 
each curve component is the normalization of a curve defined by an equation. As we will see, our specialization method is {\it canonically} prescribed by the tropicalization map, whereas the usual algebraic approach to specialization relies on the use of judiciously-chosen semistable models over the valuation ring. While constructing explicit algebraic models is easy when the generic fiber is {\it general} in moduli, it is more delicate when the generic fiber has special geometry, as is the case here. We expect that our construction should be useful more generally for probing the geometry of linear series on curves embedded as complete intersections in toric varieties.

\medskip
In the remainder of this paper, we will make use of a number of formal conventions. A {\it linear series of degree $d$ and rank $r$}, or $g^r_d$, on an algebraic curve $X$ consists of a line bundle $L$ of degree $d$ together with an $(r+1)$-dimensional vector subspace $V \sub H^0(X,L)$ of holomorphic sections, for some $r \geq 0$; it is {\it complete} when $V=H^0(X,L)$.
Given a real hyperelliptic algebraic curve $X$ of 
genus $g\,\geq\,1$ with non-empty real part, $[D]\,\in \,\mbox{Pic}^2 X$ will denote
the pullback on $X$ of a real point from $\mathbb{P}^1$ under the hyperelliptic structure map. 

\medskip
Our main quantitative results relate to the total number of real inflection points of the complete real linear series $|\mathcal{L}_{\mathbb{R}}(kD)|$; the answer depends on how large $k$ is relative to $g$.

\medskip
Whenever $1\,\leq \,k \,\leq\, g$, 
the classical Pl\"ucker formula implies that the total number of inflection points of the complexification $|\mathcal{L}(kD)|$ is $w_\mathbb{C}(g,k):=k(k+1)(g+1)$. 
In {\bf Theorem~\ref{realpartcaseg}} we give a complete characterization of the distribution of
the real inflection points of $|\mathcal{L}_\mathbb{R}(kD)|$ along the connected components of 
the real part of $X$ whenever $1\,\leq\, k \,\leq\, g$. In particular, we deduce the following result.

\begin{theorem}
Let $X$ be a real hyperelliptic curve of genus $g\,\geq\,1$ with $n(X)\,\geq\,1$ real connected
components and suppose that $1\,\leq\, k \,\leq\, g$. The real linear series $|\mathcal{L}_\mathbb{R}(kD)|$ then
has precisely $\frac{n(X)}{g+1}w_\mathbb{C}(g,k)\,=\,k(k+1)n(X)$ real inflection points.
\end{theorem}

On the other hand, whenever $k > g$, the real complete linear series $|\mathcal{L}_{\mathbb{R}}(kD)|$ is a $g^{2k-g}_{2k}$ and the total number of inflection points of the complexification $|\mathcal{L}(kD)|$ is $w_\mathbb{C}(g,k):=g(2k-g+1)^2$. In this case, {\bf Theorem~\ref{inflection_along_ramification_divisor}} gives  a lower bound for the number of real inflection points.

\begin{theorem}\label{lower_reality_bound_for_hyperell_curve} Let $X$ be a real hyperelliptic curve of genus $g\,\geq\,1$ with $n(X)\,\geq\,1$ real connected
components and let $k>g$.
The real linear series $|\mathcal{L}_\mathbb{R}(kD)|$ has at least $g(g+1)n(X)$ real inflection points.
\end{theorem}
The previous result is explained by a more general fact that we prove via a local analysis of vanishing orders of holomorphic sections of $|\mathcal{L}_\mathbb{R}(kD)|$, namely that each point of the ramification locus of the hyperelliptic cover $X \ra \mb{P}^1$ has inflectionary weight $\binom{g+1}{2}$.

\medskip
The remainder of our results are related to our Jonsson and Viro-based construction of limit
linear series on a (marked) metrized complex of elliptic curves for the case $k>g$. In
{\bf Lemma~\ref{specialization_of_linear_series}}, we show how the Jonsson--Viro degeneration
produces a $g^{2k-g}_{2k}$ along each elliptic component $E_i$, $i=1,\dots,g$, of the metrized
complex. In {\bf Theorem~\ref{inflection_along_ell_curves}} we calculate the vanishing
sequences for holomorphic sections in each of four marked points along $E_i$, including
those {\it points of attachment} corresponding to edges linking neighboring elliptic
components. One upshot (see {\bf Corollary~\ref{compatibility}}) of
Theorem~\ref{inflection_along_ell_curves} is that our limit linear series satisfy the natural
analogue of the compatibility condition for vanishing sequences in points of attachment that
characterizes Eisenbud--Harris limit linear series. 

\medskip
Another is an explicit calculation of the
inflectionary weight contributed by the marked points along each $E_i$. Further,
in {\bf Theorem~\ref{convergence}}, we prove a regeneration-type result that specifies
precisely how the inflection divisor of the complete linear series $|\mc{L}_{\mb{R}}(kD)|$ along
the hyperelliptic curve $X$ is related to the inflection divisors of the series $g^{2k-g}_{2k}$
along the elliptic components $E_i$. Regeneration depends crucially on the {\it openness} of the projection of the Jonsson--Viro degeneration from the total space to the underlying parameter space in an analytic neighborhood of the central fiber. 

\medskip
Finally, Theorems~\ref{inflection_along_ramification_divisor} and \ref{convergence}  may be used to produce complete linear series on real hyperelliptic curves with {\it controlled} real inflection. We make this explicit in {\bf Theorem~\ref{lower_bounds_for_reality}}. 

\medskip
The roadmap of this paper is as follows. In Section \ref{FactsandDefinitions}, we recall some basic facts about real linear series on a real algebraic curve and their inflection points, which determine a corresponding inflection divisor. In Section \ref{RealHyperCurves} we recall some basic facts about real hyperelliptic curves, and given a real hyperelliptic curve $X$ of genus $g\geq1$ with non-empty real part and a $[D] \in \mbox{Pic}^2(X)$ a real $g^1_2$, we characterize the real inflectionary degree of $|\mathcal{L}_\mathbb{R}(kD)|$ whenever $k \leq g$, and establish some general properties of the real inflection locus of $|\mathcal{L}_\mathbb{R}(kD)|$ whenever $k > g$. In Sections \ref{Jonsson-Viro} and \ref{Section_Specialization} we use tropicalization techniques to enhance Viro's patchworking and to define a canonical specialization, by which we associate a metrized complex of curves to any given family of plane hyperelliptic curves. 
Finally in Section \ref{Combinatorial_construction} we give a combinatorial method to construct real hyperelliptic curves of genus $g\geq2$ with controlled numbers of real inflection points of the real linear series $|\mathcal{L}_\mathbb{R}(kD)|$ for $k\geq g+1$.

{\fl \it Acknowledgement.} The second author thanks the Tata Institute for hospitality and support in February-March 2017, during which collaboration on this project started. All three authors thank the anonymous referee for his or her comments, which have led to improvements in the exposition.

\section{Fundamental facts and definitions}\label{FactsandDefinitions}

In what follows a {\it real algebraic variety} $X$ denotes a pair $X\,=\,(X_\mathbb{C},\,\sigma_X)$ 
consisting of a complex algebraic variety $X_\mathbb{C}$ together with an anti-holomorphic 
involution $\sigma_X\,:\,X_\mathbb{C}\longrightarrow\,X_\mathbb{C}$. Equivalently, a real algebraic 
variety is any complex algebraic variety of the form $X=X_\mathbb{R}\otimes_\mathbb{R}\mathbb{C}$,
where $X_{\mathbb{R}}$ is a scheme defined over $\mathbb{R}$. We denote by $X(\mathbb{R})=(X_\mathbb{C})^{\sigma_X}$ the real 
part of $X$ and by $n(X)$ the number of connected components of $X(\mathbb{R})$.

\medskip
A {\it morphism} between real algebraic varieties $(X_\mathbb{C},\,\sigma_X)$ and $(Y_\mathbb{C},\,\sigma_Y)$ is
a morphism $$f\,:\,X_\mathbb{C}\,\longrightarrow\, Y_\mathbb{C}$$ of complex algebraic varieties compatible with the corresponding anti-holomorphic involutions, i.e. such that $f\circ\sigma_X=\sigma_Y\circ f$.

\medskip 
Hereafter we restrict our attention to smooth real algebraic curves $X\,=\,(X_\mathbb{C},\,
\sigma_X)$ with $X(\mathbb{R})\,\neq\,\emptyset$. In that case, any $\sigma_X$-invariant divisor on
$X$ is of the form
\begin{equation}
\label{sigma-invariant_divisors}
D=\sum_{p\in X(\mathbb{R})}n_p\cdot p+\sum_{p\notin X(\mathbb{R})}n_p\cdot(p+\sigma_X(p)).
\end{equation}
The first summand in \eqref{sigma-invariant_divisors} is the {\it real part} of $D$, and we denote its degree as $\text{deg}_\mathbb{R}(D)$. We say that $D$ is {\it totally real} if it coincides with its real part. 

\medskip
Moreover, as explained in \cite{GH}, the real part $\text{Pic }X(\mathbb{R})$ of $\mbox{Pic } X$ is precisely the set of linear equivalence classes represented by a $\sigma_X$-invariant divisor. Let $S_1,\ldots,S_{n(X)}$ be the connected components of $X(\mb{R})$. There is a corresponding {\it parity homomorphism} 
\[
c:\text{Pic }X(\mathbb{R})\longrightarrow(\mathbb{Z}/2\mathbb{Z})^{n(X)}
\]
defined by 
\[
[D]\longmapsto (\text{deg}(D|_{S_1})\mod2,\ldots,\text{deg}(D|_{S_{n(X)}})\mod2).
\]
An obvious but nonetheless useful fact is that the parity of $\text{deg}(D)$ is the parity of the sum of the components of its parity vector $c(D)$.

\medskip
The {\it topological type} of $X$ is the triple $(g(X),n(X),a(X))$, where as usual 
$a(X)$ is one or zero depending on whether $X(\mb{C}) \setminus X(\mb{R})$ is connected or disconnected, respectively.

\medskip
Given a $\sigma_X$-invariant divisor $D$, we denote by $\mathcal{L}_\mathbb{R}(D)$ the real algebraic line bundle defined by $D$, and we denote its complexification $\mathcal{L}_\mathbb{R}(D)\otimes_\mathbb{R}\mathbb{C}$ simply by $\mathcal{L}(D)$. Then 
$H^0(X_{\mathbb{R}},\, \mathcal{L}_\mathbb{R}(D))$ is a real vector space satisfying 
\[
H^0(X_{\mathbb{R}},\, \mathcal{L}_\mathbb{R}(D))\otimes_{\mathbb{R}}\mathbb{C}\,=\, H^0(X_\mathbb{C},\, \mathcal{L}(D)).
\]
A {\it real linear series} (of degree $d$ and rank $r$) on $X$ is a pair $(L_\mathbb{R},V_\mathbb{R})$ consisting of an algebraic line bundle $L_\mathbb{R}\in \text{Pic }X(\mathbb{R})$ of degree $d$ and a real vector subspace $V_\mathbb{R}\subseteq H^0(X_{\mathbb{R}},L_\mathbb{R})$ of dimension $r+1 \geq 1$. The {\it inflection divisor } associated to $(L_\mathbb{R},V_\mathbb{R})$ is the divisor $\text{Inf}(L_\mathbb{R},V_\mathbb{R})=\sum_{p\in X}|p|\cdot p$, where $|p|$ is the inflectionary weight of $(L_\mathbb{R},V_\mathbb{R})$ at $p$; it is an effective $\sigma_X$-invariant divisor of degree $(r+1)(d+r(g-1))$, where $g$ is the genus of $X$.

\medskip
Let $(\mathcal{L}_\mathbb{R}(D),V_\mathbb{R})$ be a real linear series of degree $d$ and rank $r$ on $X$. Associated to a basis $\mathcal{F}=\{f_0,\ldots,f_r\}$ of $V_\mathbb{R}$, there is a (real) section $\text{Wr}(\mathcal{F})\in H^0(X_{\mb{R}},\mathcal{L}_\mathbb{R}((r+1)D+\binom{r+1}{2}K_X))$ called the {\it Wronskian} of $\mathcal{F}$. The divisor $\text{div Wr}(\mathcal{F})$ is independent of the choice of the basis for $V_\mathbb{R}$, and coincides with the inflection divisor of $(\mathcal{L}_\mathbb{R}(D),V_\mathbb{R})$.

\medskip
Hereafter, given a hyperelliptic curve $X$ over a field $F$,
$$\pi\,:\,X\,\longrightarrow \,F\mathbb{P}^1$$ will denote the two-sheeted branched cover obtained from its $g_2^1$, and $R_\pi$ will denote the ramification divisor of $\pi$. We will also abusively omit including either $X_{\mb{R}}$ or $X_{\mb{C}}$ in our notation for spaces of (real or holomorphic) sections of line bundles whenever the choice is clear from the context.

\section{Real linear series on real hyperelliptic curves}\label{RealHyperCurves}

Let $X\,=\,(X_\mathbb{C},\,\sigma_X)$ be a real hyperelliptic curve of genus $g\geq 1$ and let $[D]
\,\in \,\text{Pic }X$ be a $g_2^1$ on $X$. If $X(\mathbb{R})\,\neq\,\emptyset$, then $[D]
\,\in\,\text{Pic }X(\mathbb{R})$ and $\mathbb{P}(H^0(\mathcal{L}(D)))\,\cong\,(\mathbb{CP}^1,\sigma_{\mathbb{CP}^1})$, where $\sigma_{\mathbb{CP}^1}$ is the 
real structure $[z_1:z_2]\longmapsto[\overline{z_1}:\overline{z_2}]$ given by conjugation. It follows that the two-sheeted branched cover $\pi:X_\mathbb{C}\longrightarrow\mathbb{CP}^1$ obtained from $[D]$ is defined over $\mathbb{R}$.

\medskip
The function field $K(\mathbb{CP}^1,\sigma_{\mathbb{CP}^1})$ of $(\mathbb{CP}^1,\sigma_{\mathbb{CP}^1})$ is of the form 
\begin{equation}\label{function_field_P^1}
K(\mathbb{CP}^1,\sigma_{\mathbb{CP}^1})=\mathbb{R}(x);
\end{equation}
with respect to the identification \eqref{function_field_P^1}, the function field of $X$ is of the form 
\[
K(X)=\mathbb{R}(x)[y]/(y^2-f(x))
\]
for some separable polynomial $f(x)\in\mathbb{R}[x]$ of degree $2g+1$ or $2g+2$. We suppose that $2g$ branch points of the map $\pi$ lie in $\mathbb{C}^*$ and that $X$ has $n(X)$ real connected components; here $1\leq n(X)\leq g+1$. Let $U\subset\mathbb{C}^2$ denote the real affine plane curve defined by
\begin{equation}\label{HyperellipticCurve}
y^2=f(x)=x\prod_{i=1}^{2n(X)-2}(x-p_i)\prod_{j=1}^{g+1-n(X)}(x-q_j)(x-\overline{q_j})
\end{equation}
where the points $p_i\in\mathbb{R}^*$ and $\:q_j\in\mathbb{C}\setminus\mathbb{R}$ are all distinct. 

\begin{rem}\label{comp}
Hereafter we will denote by $U$ the real affine plane curve \eqref{HyperellipticCurve}, and by $X$ its compactification $U\cup\{\infty\}$ in $\mathbb{CP}^1\times\mathbb{CP}^1$. We use $\infty$ to denote the preimage of the point $\infty \in \mb{P}^1$ that is distinguished by our choice of affine coordinates. As we are assuming the real structure on $X$ is given by $(x,y)\mapsto(\overline{x},\overline{y})$, the topological type $(g,n(X),a)$ of the real hyperelliptic curve $X$ of genus $g\geq2$ is $(g,n(X),0)$ whenever $n(X)=g+1$, and $(g,n(X),1)$ otherwise \cite{C}.
\end{rem}

Now consider the divisor $D'=\infty$ in $(\mathbb{CP}^1,\sigma_{\mathbb{CP}^1})$; we have
\[
[D]=[\pi^*(D^{\prime})]=[2 \cdot \infty] \text{ and }[K_{X}]=[(g-1)\pi^{\ast}(D^{\prime})]=[(2g-2) \cdot \infty]
\]
and for $k\geq g$, the space $H^0(\mathcal{L}_\mathbb{R}(kD))$ is a real vector space of 
dimension $r+1\,=\,2k-g+1$ with a real basis $\mathcal{F}\,=\,\{f_0,\ldots,f_{2k-g}\}\,\subset\, K(X)$ given 
by $f_i\,=\,x^i$ for $0\,\leq\, i\,\leq\, g$ when $k\,=\,g$, and by $f_i\,=\,x^i$ for
$0\,\leq\, i\,\leq\, k$ and $f_{i}\,=\,x^{i-k-1}y$ for $k+1\,\leq\, i\,\leq\, 2k-g$ whenever $k\,>\,g$.

\medskip
The basic theory of Section~\ref{FactsandDefinitions} implies that for a certain real rational function $h$ on $X$, the divisor of the Wronskian associated to the basis $\mathcal{F}$ of $H^0(\mathcal{L}_\mathbb{R}(kD))$ is of the form 
\[
\begin{aligned}
\text{div Wr}(\mathcal{F})&=(r+1)kD+\tfrac{r(r+1)}{2}K_{X}+\text{div}_{X}(h)\\
&=(r+1)(2k+r(g-1))\cdot\infty+\text{div}_{X}(h).
\end{aligned}
\]
It follows that the Wronskian divisor is effective, $\sigma_X$-invariant, and of degree 
\[
(r+1)(2k+r(g-1)) \,=\,g(2k-g+1)^2
\]
on $X$.

\medskip
Note that since $\text{div Wr}(\mathcal{F})$ is effective, the pole divisor
$\text{div}_\infty(h)$ of $h$ is supported on $\{\infty\}$. In particular, there exists a real
regular function $\alpha$ on $U$ such that $\text{div}_0(h)\,=\,\text{div}_U(\alpha)$; and further,
since $U\subset\mathbb{C}^2$, we can choose a representative for $\alpha$ that is regular
on $\mathbb{C}^2$.

\medskip
In other words, there exists some $\alpha\,\in\,\mathbb{R}[x,y]$ in terms of which the
inflection divisor of the complete real linear series $|\mathcal{L}_\mathbb{R}(kD)|$ on $X$
may be realized as 
\begin{equation}
\label{InflectionDivisor}
\text{Inf}(|\mathcal{L}_\mathbb{R}(kD)|)=\text{div}_U(\alpha)+m\cdot\infty=[U\cap V(\alpha)]+m\cdot\infty
\end{equation}
where $m=g(2k-g+1)^2-\text{deg div}_U(\alpha)\geq0$ and $[U\cap V(\alpha)]$ is the divisor associated to the intersection scheme $U\cap V(\alpha)$. In particular, the real part of $\text{Inf}(|\mathcal{L}_\mathbb{R}(kD)|)$ consists of the real part of $[U\cap V(\alpha)]$ together with $m\cdot\infty$.

\medskip
Now let $\sigma_h:X\longrightarrow X$ be the hyperelliptic involution sending $(x,y)$ to $(x,-y)$, and let $G\subset\text{Aut}(X)$ be the subgroup generated by $\sigma_X,\sigma_h$, which is isomorphic to $\mathbb{Z}/2\mathbb{Z}\times \mathbb{Z}/2\mathbb{Z}$, as $\sigma_X\sigma_h
\,=\,\sigma_h\sigma_X$.

\begin{prop}
\label{troispartiesdiv}
Let $k\geq g$. The inflection divisor $\text{Inf}(|\mathcal{L}_\mathbb{R}(kD)|)$ is the sum of two effective divisors with disjoint supports 
\[
\text{Inf}(|\mathcal{L}_\mathbb{R}(kD)|=R+S
\]
where $R$ is supported on the ramification locus of $\pi$ and $S$ is $G$-invariant. If $k=g$, then $S=0$.
\end{prop}

\begin{proof}
Let $\alpha$ be 
as in \eqref{InflectionDivisor}. Since $\tfrac{\partial(y^2-f)}{\partial y}\,=\,2y$, it follows
that in the open subset $U_y\,:=\,U\setminus V(y)$ the restriction $\alpha|_{U_y}$ is given by $\alpha|_{U_y}\,=\,\text{det}(f_i^{(j)})_{0\leq i,j\leq 2k-g}$, where $f_i^{(j)}\,=\,
\tfrac{\partial^jf_i}{\partial x^j}$. Note that $\text{div}_U(\alpha)\,=\,
\text{div}_{U_y}(\alpha|_{U_y})+R'$, where $R'$ is a divisor supported on the closed subset
$V(y)\cap U$. We need to show that $S=\text{div}_{U_y}(\alpha|_{U_y})$ is $G$-invariant.

\medskip
The simplest situation occurs when $k\,=\,g$. In that case $S=0$ since $\alpha|_{U_y}\,=\,1$, and it follows that
\[
\text{Inf}(|\mathcal{L}_\mathbb{R}(gD)|)\,=\,R'+m\cdot\infty\, .
\]
Now assume that $k>g$. We then have 
\begin{equation}
\label{alpha_in_U}
\alpha|_{U_y}\,=\,\text{det}\left(\begin{array}{cccc}
 (x^0y)^{(k+1)}& (x^0y)^{(k+2)} &\cdots &(x^0y)^{(2k-g)}\\
 (x^1y)^{(k+1)}& (x^1y)^{(k+2)} &\cdots &(x^1y)^{(2k-g)}\\
 &&&\\
 (x^{k-g-1}y)^{(k+1)}& (x^{k-g-1}y)^{(k+2)} &\cdots &(x^{k-g-1}y)^{(2k-g)}\\
\end{array}\right)
\end{equation}
which is a square $(k-g) \times (k-g)$ matrix. In the open set $U_y$ we have 
\[
y^{\pr}\,=\,\tfrac{f'}{2y}\,=\,\left(\tfrac{f'}{2f}\right)y\,:=\,P_{0,1}(x)y
\]
and it follows by induction that for $j\,>\,1$ we have $y^{(j)}\,=\,P_{0,j}(x)y$ for some $P_{0,j}(x)
\,\in\,\mathbb{R}(x)$. It follows that for $i,\,j\,\geq\,0$,
we have $(x^iy)^{(j)}\,=\,P_{i,j}(x)y$ for some $P_{i,j}(x)\,\in\,\mathbb{R}(x)$. The upshot is
that there exists some $Q(x)\,\in\,\mathbb{R}(x)$ for which 
\begin{equation*}
\alpha|_{U_y}\,=\,Q(x)y^{k-g}\, .
\end{equation*}
Since $y\neq0$ on $U_y$, it follows that the divisor $\text{div}_{U_y}(\alpha|_{U_y})$ is determined
by the divisor $\text{div}_{U_f}(Q)$ of $Q$, which is
a real regular function on the open set 
$U_f\,:=\,\mathbb{C}\setminus{V(f)}$.

\medskip
Now suppose that 
\[
\text{div}_{U_f}(Q)\,=\,\sum_{p\neq p_i}n_p\cdot p+\sum_{q\neq q_j}n_q\cdot (q+\overline{q})\,;\]
then each $p$ lifts to $(p,\sqrt{f(p)})+\sigma_h((p,\sqrt{f(p)}))$ on $\text{div}_{U_y}(\alpha|_{U_y})$, while each
$q+\overline{q}$ lifts to $$(q,\sqrt{f(q)})+\sigma_X(q,\sqrt{f(q)})+\sigma_h(q,\sqrt{f(q)})+
\sigma_h\sigma_X(q,\sqrt{f(q)})\, .$$ It follows immediately that $\text{div}_U(\alpha|_U)$ is $G$-invariant.
\end{proof}
 
\begin{thm}
 \label{realpartcaseg}
 When $1\leq k \leq g$, the inflection divisor $\text{Inf}(|\mathcal{L}_\mathbb{R}(kD)|)$ is $\binom{k+1}{2}$ times the ramification divisor of $\pi$. In particular, the linear series $|\mathcal{L}_\mathbb{R}(kD)|$ has $k(k+1)n(X)$ real inflection points.
 \end{thm}
 
\begin{proof}
To simplify the exposition we focus on the case $k\,=\,g$; the extension to the seemingly more 
general case $k \,\leq\, g$ is easy, and will be described at the end. We saw in Proposition 
\ref{troispartiesdiv} that $\text{Inf}(|\mathcal{L}_\mathbb{R}(gD)|)$ is supported along the 
ramification locus of $\pi$. Let $\alpha$ be as in \eqref{InflectionDivisor}. Since 
$\tfrac{\partial(y^2-f)}{\partial x}\,=\,f^{\pr}(x)$, it follows that on the open subset 
$U_x\,=\,U\setminus\{(p,\pm\sqrt{f(p)})\::\:f^{\pr}(p)=0\}$ the restriction $\alpha|_{U_x}$ is given by $\alpha|_{U_x}\,=\,\text{det}(f_i^{(j)})_{0\leq i,j\leq g}$, where 
$f_i^{(j)}\,=\,\tfrac{\partial^jf_i}{\partial y^j}$.
 
 \medskip
Write $\mf{D}^j\,=\,\tfrac{\partial^j}{\partial y^j}$; then, since $f_0\,=\,1$, we have 
$\alpha|_{U_x}\,=\,\text{det}(\mc{D}^j(x^i))_{1\leq i,j\leq g}$. Let $R_2\,=\,(\mf{D}^j(x^2))_{1\leq 
j\leq g}$ be the vector corresponding to the second row of the matrix $(\mf{D}^j(x^i))_{1\leq 
i,j\leq g}$. It may be written as $2x\cdot R_1+R_2^{(1)}$, where 
$R_2^{(1)}\,=\,(r_{2,j}^{(1)})_{1\leq j\leq g}$ is the vector with entries $r_{2,1}^{(1)}\,=\,0$ and 
$r_{2,j}^{(1)}\,=\,2\sum_{k=1}^{j-1}\binom{j-1}{k}\mf{D}^k(x)\mf{D}^{j-k}(x)$ for $j\,\geq\,2$. In 
particular, we have $r_{2,2}^{(1)}\,=\,2(\mf{D}^1(x))^2$.
 
 \medskip
 Similarly, let $R_3\,=\,(\mf{D}^j(x^3))_{1\leq j\leq g}$ be the third row. It may be written as
$3x^2\cdot R_1+R_3^{(1)}$, where $R_3^{(1)}=(r_{3,j}^{(1)})_{1\leq j\leq g}$ is the vector with entries
$$r_{3,1}^{(1)}\,=\,0 \ \ \text{and}\ \ r_{3,j}^{(1)}
\,=\,3\sum_{k=1}^{j-1}\binom{j-1}{k}\mf{D}^k(x^2)\mf{D}^{j-k}(x)\, \ \text{ for } \ j\,\geq\,2\, .$$
This time we have $R_3^{(1)}=3xR_2^{(1)}+R_3^{(2)}$, where $r_{3,1}^{(2)}=r_{3,2}^{(2)}=0$ and
$$r_{3,j}^{(2)}\,=\,6\sum_{k=2}^{j-1}\sum_{\ell=1}^{k-1}\binom{j-1}{k}\binom{k-1}{\ell}\mf{D}^{j-k}
(x)\mf{D}^{k-\ell}(x)\mf{D}^\ell(x)$$ for $j\,\geq\,3$. So $R_3\,=\,3x^2\cdot R_1+3xR_2^{(1)}+R_3^{(2)}$, with
$r_{3,3}^{(2)}\,=\,6(\mf{D}^1(x))^3$.
 
 \medskip
 In general, we may write the $i$-th row $R_i$ as a linear combination
\[
\alpha_{i,1}R_1+\alpha_{i,2}R_2^{(1)}+\cdots+\alpha_{i,i-1}R_{i-1}^{(i-2)}+R_i^{(i-1)}
\]
in which the vector $R_i^{(i-1)}\,=\,(r_{i,j}^{(i-1)})$ satisfies $r_{i,j}^{(i-1)}\,=\,0$ for
$1\,\leq\, j\,\leq\, i-1$ and $r_{i,i}^{(i-1)}\,=\,\alpha_{i,i}(\mf{D}^1(x))^i$ for some
$\alpha_{i,i}\,\in\,\mathbb{R}\setminus\{0\}$.
 
 \medskip
 Consequently, we have 
 \[
 \alpha|_{U_x}\,=\,\text{det}((x^i)^{(j)})_{1\leq i,j\leq g}\,=\,a(x^{\pr})^{\binom{g+1}{2}}
 \]
for some $a\in\mathbb{R}\setminus\{0\}$, and the theorem for $k\,=\,g$ follows because
$x^{\pr}\,=\,\tfrac{2}{f^{\pr}}y$ on $U_x$.
 
 \medskip
Finally, if $1\,\leq\, k \,\leq\, g$, then $H^0(\mathcal{L}_\mathbb{R}(kD))$ is spanned by
$\{x^0,\ldots,x^k\}$, and $\text{Inf}(|\mathcal{L}_\mathbb{R}(kD)|)$ is computed exactly as
above. In particular, we have
\[
 \alpha|_{U_x}\,=\, a(x^{\pr})^{\binom{k+1}{2}}
\]
for some $a\,\in\,\mathbb{R}\setminus\{0\}$. This completes the proof.
\end{proof}
 
\begin{rem}
When $X$ is hyperelliptic, the inflection divisor of its canonical series is
$\binom{g}{2}$ times the ramification divisor of $\pi$ \cite[p.~274]{GriH}. This is in agreement
with our result, as $[K_X]\,=\,(g-1)[D]$. In Theorem~\ref{inflection_along_ramification_divisor}, we
will prove a more precise version of Proposition~\ref{troispartiesdiv} when $k\geq g$ via a local analysis of vanishing orders
of sections of $|\mc{L}_{\mb{R}}(D)|$.
\end{rem}

\begin{rem}
When $k\,>\,g$, we have
 \[
 \alpha|_{U_y}\,=\,\text{det}\biggl((k+j)!\frac{y^{(k+1+j-i)}}{(k+1+j-i)!}\biggr)_{1\leq i,j\leq k-g}
 \]
 and
 \[
 \alpha|_{U_x}\,=\,\text{det}(a_{i,j})_{1\leq i,j\leq 2k-g-1}\, ,
 \]
 where 
 \begin{equation*}
 a_{i,j}\,=
 \begin{cases}
 (x^i)^{(j+1)},&\text{ if }1\leq i\leq k,\\
 (j+1)(x^{i-k})^{(j)},&\text{ if }k+1\leq i\leq 2k-g-1.\\
 \end{cases}
 \end{equation*}
Note that $\alpha|_U$ is in fact the determinant of a square Toeplitz matrix. It would be interesting to identify these functions explicitly.
 \end{rem}
 
 We end this section by commenting on the case $g=1$, which was resolved in \cite{G} (see also Section \ref{Combinatorial_construction}). In this situation, the real components of the hyperelliptic curve defined by (compactifying) \eqref{HyperellipticCurve}, along with the real inflection points of $|\mathcal{L}_\mathbb{R}(kD)|$, obey the following dichotomy. 
 \begin{enumerate}
 \item[1.] $n(X)\,=\,1$ if and only if 
 $f$ has conjugate non-real roots $q_1, \overline{q_1} \,\in\,\mathbb{C}\setminus\mathbb{R}$. In this situation, the inflection divisor $\text{Inf}(|\mathcal{L}_\mathbb{R}(kD)|)$ has $2k$ (distinct) real points.
 \item[2.] $n(X)\,=\,2$ if and only if $f$ has distinct roots $p_1,p_2 \,\in\,\mathbb{R}^*$. Since the parity vector $c(kD)$ of $kD$ is $(0,0)$, the inflection divisor $\text{Inf}(|\mathcal{L}_\mathbb{R}(kD)|)$ has $2k$ real points on each connected component of $X(\mathbb{R})$.
 \end{enumerate}
 
 In the remainder of this work we will focus on the case $k>g>1$.

\section{Enhancing Viro's patchworking construction}
\label{Jonsson-Viro}
Viro's patchworking method is a tool for constructing real plane algebraic curves with controlled topology. In this section we will apply an enhanced version of this method to construct useful one-parameter families of real affine plane hyperelliptic curves of the form \eqref{HyperellipticCurve}.

\medskip
Any such family may be viewed as a plane curve defined over a non-Archimedean field, and so is associated with a subdivision of a lattice triangle. More precisely, for a fixed choice of $g\geq2$, let $\Delta \sub \mathbb{R}^2$ denote the lattice triangle $\Delta=\text{Conv}\{(0,2),(1,0),(2g+1,0)\}$. Let $\Theta$ denote the  subdivision of $\Delta$ whose 2-dimensional faces are the triangles 
\begin{equation}
\Theta_j\,=\,\text{Conv}\{(0,2),(2j-1,0),(2j+1,0),\:j=1,\ldots,g.\}
\end{equation}

\medskip
This subdivision is regular and we can construct it using a function $\nu:[1,2g+1]\cap\mathbb{Z}\longrightarrow \mathbb{Z}_{\geq0}\cup\{\infty\}$ in such a way that the convex hull of its graph induces the regular subdivision $\{[2j-1,2j+1]\::\:j=1,\ldots,g\}$ on the interval $[1,2g+1]$, then the convex hull of the set $\{(0,2,0),(i,0,\nu(i)),i=1,\ldots,2g+1\}$ induces the  subdivision $\Theta$ on $\Delta$. We use these values to define the following patchworking polynomial 
\begin{equation}\label{pre_patchworked_curve}
y^2-f:=t^0y^2-\sum_{i=1}^{2g+1}a_{i}t^{\nu(i)}x^i\in \mathbb{R}[t^{\pm1}][x,y]
\end{equation}
where $a_i=0$ if and only if  $\nu(i)=\infty$. In particular we have $a_{2i-1}\neq0$ for $i=1,\ldots,g+1$.

\medskip
Our present aim is to show how any family of hyperelliptic curves embedded as a hypersurface in a 
toric variety as in \eqref{pre_patchworked_curve} naturally has an associated limit object in the 
category of the {\it metrized complexes} studied in \cite{AB}. To do so, we adapt Berkovich's 
construction \cite{Be} of a {\it hybrid} family of topological spaces that interpolates between Archimedean 
and non-Archimedean analytifications of a given complex algebraic variety. Our presentation 
follows closely that of Jonsson \cite{J}.

\begin{rem}
In this section, when treating families of {\it algebraic} schemes, we work over the non-Archimedean field $\mathbb{K}=\mathbb{C}((t))^\text{alg}$ endowed with the $t$-adic norm $v_t$ normalized to satisfy $v_t(t)=e^{-1}$. 
In every {\it Berkovich-analytic} argument we make we work over the completion $\widehat{\mathbb{K}}$, which remains algebraically closed. However, in the interest of not overburdening notation we will continue to write $\mb{K}$ in place of $\widehat{\mb{K}}$ in these cases.
\end{rem}

\medskip
To begin, set $\mathcal{U}:=\text{Spec}(\mathbb{R}[t^{\pm1}][x,y]/(y^2-f)\otimes_{\mathbb{R}}\mathbb{C})$ and $\mathbb{G}_m:=\text{Spec}(\mathbb{R}[t^{\pm1}]\otimes_{\mathbb{R}}\mathbb{C})$. The inclusion of rings
\[
\mathbb{R}[t^{\pm1}]\otimes_{\mathbb{R}}\mathbb{C}\hookrightarrow \mathbb{R}[t^{\pm1}][x,y]/(y^2-f)\otimes_{\mathbb{R}}\mathbb{C}
\]
induces a surjective real morphism $p\,:\,\mathcal{U}\,\longrightarrow\,\mathbb{G}_m$ of real algebraic varieties. Let 
\[
\mathcal{U}_{\mb{K}}\,:=\,\mathcal{U}\times_{\mathbb{G}_m}\mb{K}
\,=\,\text{Spec}(\mb{K}[x,y]/(y^2-f))
\]
denote the $\mb{K}$-curve obtained from the $\mb{G}_m$-surface $\mc{U}$ via the obvious base change.

\medskip 
Now let $v_0,\,v_\infty\,:\,\mathbb{C}\,\longrightarrow\,\mathbb{R}_{\geq 0}$ denote the trivial
and the Archimedean norms of $\mathbb{C}$ respectively, and consider the function 
$v_\text{hyb}\,:\,\mathbb{C}\,\longrightarrow\,\mathbb{R}_{\geq0}$ defined by 
\[
z\,\longmapsto\,\text{max}\{v_0(z),v_\infty(z)\}\, .
\]
This is a sub-multiplicative norm on $\mathbb{C}$ and $(\mathbb{C},v_\text{hyb})$ is a Banach
ring; its Berkovich spectrum $\mathcal{M}(\mathbb{C},v_\text{hyb})$ is thus well-defined, and is in
fact homeomorphic to the unit interval $[0,\,1]$. Further, there is an analytification
functor $\text{an}(-,v_\text{hyb})$ from the category of complex algebraic varieties to the category
of $\mathcal{M}(\mathbb{C},v_\text{hyb})$-analytic spaces.
 
\medskip
Let $\text{an}(\mathcal{U},v_\text{hyb})$ and $\text{an}(\mathbb{G}_m,v_\text{hyb})$ be the analytification of $\mathcal{U}$ and $\mathbb{G}_m$ with respect to $v_\text{hyb}$. Let 
\[
 \mathcal{U}^\#\,:=\,\{\rho\in \text{an}(\mathcal{U},v_\text{hyb})\::\:\rho(t)\,=\,e^{-1}\}
 \]
 and 
 \[
 \mathbb{G}_m^\#:=\{\rho\in \text{an}(\mathbb{G}_m,v_\text{hyb})\::\:\rho(t)=e^{-1}\}.
 \]
 Note that as a topological space, $\mathbb{G}_m^\#$ is precisely
 the Archimedean closed disk $D_{e^{-1}}\,=\,\{z\in\mathbb{C}\::\:v_\infty(z)\leq e^{-1}\}$.
 
 \medskip
 The fiberwise behavior of the restriction $p^\#\,:\,\mathcal{U}^\#\,\longrightarrow\,
D_{e^{-1}}$ to $\mathcal{U}^\#$ of the hybrid analytification
 \[
 \text{an}(p,v_\text{hyb})\,:\,\text{an}(\mathcal{U},v_\text{hyb})\,\longrightarrow\,\text{an}(\mathbb{G}_m,v_\text{hyb})
 \]
 of the projection $p\,:\,\mathcal{U}\,\longrightarrow\,\mathbb{G}_m$ is explained by the following dichotomy.

\begin{enumerate}
\item[1.] The fiber $\mathcal{U}^\#_\varepsilon$ of $p^\#$ over $\varepsilon\,\in\, D_{e^{-1}}\cap\mathbb{R}_{>0}$ is homeomorphic to the fiber above $t=\varepsilon$ of the morphism $\text{an}(p,v_\infty)$, which is none other than the usual holomorphic analytification. In other words, $\mathcal{U}^\#_\varepsilon$ is the real plane curve $V(y^2-f_\varepsilon)$, where $f_\varepsilon=f|_{t=\varepsilon}$,

\item[2.] The fiber $\mathcal{U}^\#_0$ of $p^\#$ over $0\,\in\, D_{e^{-1}}$ is $\text{an}(\mathcal{U}_{\mathbb{C}((t))},v_t)$, the analytification of the non-Archimedean plane curve $V(y^2-f)\subset\mathbb{C}((t))^2$ with respect to $v_t$, the $t$-adic norm normalized to satisfy $v_t(t)=e^{-1}$. Note that there is a natural inclusion of $\text{an}(\mathcal{U}_{\mathbb{C}((t))},v_t)$ in $\text{an}(\mathcal{U}_{\mb{K}},v_t)$ induced by the inclusion of $(\mb{C}((t)),v_t)$ in $(\mb{K},v_t)$.
\end{enumerate}
Moreover, the map $p^\#$ is open over the Archimedean closed disk $D_{\delta}$ whenever $0 \leq \delta \ll 1$.

\medskip
From a practical point of view, the map $p^\#$ will allow us to relate inflection divisors of complex hyperelliptic curves with inflection divisors of non-Archimedean hyperelliptic curves over the field of Puiseux series $\mb{K}$ with complex coefficients. 
Note that the non-Archimedean plane curve $\mathcal{U}_{\mb{K}}(\mathbb{K})\subset \mathbb{K}^2$ is smooth if and only if $f$ has $2g+1$ distinct roots in $\mathbb{K}$.
Suppose that this is the case, and let $\mathcal{X}_{\mb{K}}$ denote the compactification $\mathcal{U}_{\mb{K}}(\mathbb{K})\cup\{\infty\}$ in $\mathbb{KP}^1\times\mathbb{KP}^1$. It follows that $\mathcal{X}_{\mb{K}}$ is a hyperelliptic curve of genus $g$ over $\mathbb{K}$, with associated two-sheeted cover $\pi:\mathcal{X}_{\mb{K}}\longrightarrow\mathbb{KP}^1$.

\medskip
The analysis of (individual) complex curves carried out in Section \ref{RealHyperCurves} remains valid in this context. 
Namely, consider the divisor $D\,=\,2\cdot\infty$ on $\mathcal{X}_{\mb{K}}$. 
Just as before, the complete linear series $|\mathcal{L}(kD)|$ on $\mathcal{X}_{\mb{K}}$ for $k>g$ has a basis of global sections $\mathcal{F}=\{\phi_{0},\ldots,\phi_{2k-g}\}\subset \mathbb{R}[t^{\pm1}][x,y]$ defined by
\begin{equation}
\label{nafamily}
\phi_i\,=\,x^i\text{ for }0\leq i\leq k\text{, and } \phi_{i}\,=\,x^{i-k-1}y\text{ for }k+1\leq i\leq 2k-g.
\end{equation}

\medskip
By computing $\text{div Wr}(\mathcal{F})$, we see that there is some $\alpha\in \mathbb{R}[t^{\pm1}][x,y]$ for which
\begin{equation}
\label{InflectionDivisoroverK}
\text{Inf}(|\mathcal{L}(kD)|)=\text{div}_{\mathcal{U}_{\mathbb{K}}}(\alpha)+m\cdot\infty=[\mathcal{U}_{\mathbb{K}}\cap V(\alpha)]+m\cdot\infty
\end{equation}
where $m\,=\,g(2k-g+1)^2-\text{deg div}_{\mathcal{U}_{\mb{K}}}(\alpha)\geq0$, and $[\mathcal{U}_{\mb{K}} \cap V(\alpha)]$ is the divisor associated to the intersection scheme $\mathcal{U}_{\mb{K}} \cap V(\alpha)$ in $\mathbb{K}^2$.

\medskip
We will now analyze the $v_\text{hyb}$-analytification of the intersection scheme $\mathcal{U}_{\mb{K}}\cap V(\alpha)$ on $\mathbb{K}^2$, which is the closed subscheme defined by the ideal $(y^2-f,\alpha)$.
To this end, set 
\[
\mathcal{Z}\,:=\,\text{Spec}(\mathbb{R}[t^{\pm1}][x,y]/(y^2-f,\alpha)\otimes_{\mathbb{R}}\mathbb{C})\, ;
\]
this is a real algebraic variety, 
canonically equipped with a surjective real morphism $q\,:\,\mathcal{Z}\,\longrightarrow\,\mathbb{G}_m$ between real algebraic varieties. Let 
\[
\mathcal{Z}_{\mb{K}}\,:=\,\mathcal{Z}\times_{\mathbb{G}_m}\mb{K}\,=\,\text{Spec}(\mb{K}[x,y]/(y^2-f,\alpha))
\]
denote the $\mb{K}$-curve associated to $\mc{Z}$ by carrying out the obvious base change.

\medskip
According to our discussion above, the induced map $q^\#\,:\,\mathcal{Z}^\#\,\longrightarrow\,
D_{e^{-1}}$ has the following properties.

\begin{enumerate}
\item[1.] The fiber $\mathcal{Z}^\#_\varepsilon$ of $p^\#$ over $\varepsilon\,\in\, D_{e^{-1}}\cap\mathbb{R}_{>0}$ is 
the 0-dimensional scheme $V(y^2-f_\varepsilon,\,\alpha_\varepsilon)$, where $f_\varepsilon=f|_{t=\varepsilon}$ and 
$\alpha_\varepsilon\,=\,\alpha|_{t=\varepsilon}$.

\item[2.] The fiber $\mathcal{Z}^\#_0$ of $q^\#$ over $0\,\in\, D_{e^{-1}}$ is 
$$\text{an}(\mathcal{Z}_{\mathbb{C}((t))},\, v_t)\, ,$$ the analytification of the 0-dimensional scheme 
$V(y^2-f,\alpha)\,\subset\,\mathbb{C}((t))^2$ over the non-Archimedean field $(\mathbb{C}((t)),v_t)$.

\item[3.] For $0\,<\,\delta \,\ll \,1$, the map $q^\#$ is open above $D_{\delta}$. 
\end{enumerate}

 \medskip
 Moreover, for generic values $0\,<\,\varepsilon \,\ll \,1$, the real plane curve
$\mathcal{U}^\#_\varepsilon\,\subset\, \mathbb{C}^2$ is hyperelliptic. Let $\mathcal{X}^\#_\varepsilon$
denote its compactification $\mathcal{U}^\#_\varepsilon\cup\{\infty_\varepsilon\}$ in
$\mathbb{CP}^1\times\mathbb{CP}^1$ and $D\,=\,2\cdot\infty_\varepsilon$. The inflection divisor of
the complete real linear series $|\mathcal{L}_\mathbb{R}(kD)|$ on $\mathcal{X}^\#_\varepsilon$ is
given by
 \[
 \text{Inf}(|\mathcal{L}_\mathbb{R}(kD)|)\,=\,[\mathcal{Z}^\#_\varepsilon]+m_\varepsilon\cdot \infty_\varepsilon
 \]
 where $[\mathcal{Z}^\#_\varepsilon]$ is the divisor associated to the 0-dimensional closed subscheme $\mathcal{Z}^\#_\varepsilon\,\subset\, \mathcal{U}^\#_\varepsilon$. Similarly, the inflection divisor of the linear series $|\mathcal{L}(kD)|$ on the hyperelliptic curve $\mathcal{X}_{\mb{K}}$ is given by
 \[
 \text{Inf}(|\mathcal{L}(kD)|)\,=\,[\mathcal{Z}^\#_0]+m\cdot\infty\
 \]
 where $[\mathcal{Z}^\#_0]$ is the divisor associated to the 0-dimensional closed subscheme $\mathcal{Z}^\#_0\subset \mathcal{U}^\#_0$.
 
\medskip
Now let $\mathcal{X}^\#_0$ denote the compactification $\mathcal{U}^\#_0\cup\{\infty\}$ in $\mathbb{KP}^{1,\text{an}}\times \mathbb{KP}^{1,\text{an}}$; up to a base change, this is just the  analytification $\text{an}(\mathcal{X}_{\mb{K}},v_t)$. We can easily compute the skeleton 
\[
\text{Sk}(\mathcal{X}^\#_0)\,\hookrightarrow\, \mathcal{X}^\#_0
\]
since $\mathcal{X}_{\mb{K}}$ is smooth and proper; see Figure \ref{fig:skeleton}. A detailed discussion of the skeleton is given in the next section, in connection with the tropicalization technique; however, it already seems useful to give a concrete description here.

 \begin{figure}[!htb]
 \centering
 \includegraphics{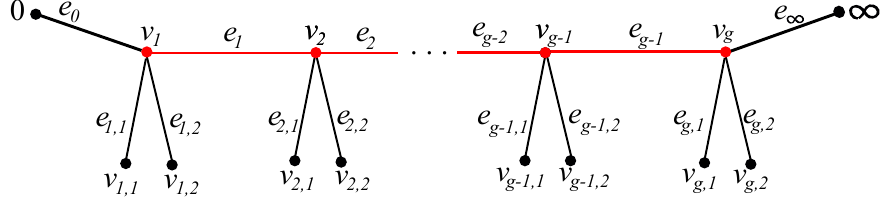}
 \caption{Skeleton of the compactification $\mathcal{X}^\#_0\,=\,
\text{an}(\mathcal{X}_{\mb{K}},v_t)$. Here $D\,=\,2\cdot\infty$.}
 \label{fig:skeleton}
 \end{figure}
 
 Explicitly, a model for the object $\text{Sk}(\mathcal{X}^\#_0)$ is the graph $G=(V,E)$ consisting of the vertices
 \[
 V\,=\,\{v_1,\ldots,v_g\}\cup\{0,\infty,v_{i,j}\::\:i=1,\ldots,g,\:j=1,2\}
 \]
 and the edges
 \[
 E\,=\,\{e_1,\ldots,e_{g-1}\}\cup\{e_0,e_\infty, e_{i,j}\::\:i=1,\ldots,g,\:j=1,2\}\, .
 \]
 We begin by describing the set $V$. The $2g+2$ vertices $\{0,\infty,v_{i,j}\::\:i=1,\ldots,g,\:j=1,2\}$ are all type I points, and this set is precisely $\text{Supp}(R_\pi)$, the support of the ramification divisor of the map 
 $\pi:\mathcal{X}_{\mathbb{K}}\longrightarrow\mathbb{KP}^1$. We thus write $V=\{v_1,\ldots,v_g\}\cup \text{Supp}(R_\pi)$.
 
 \medskip
 The $g$ vertices $\{v_1,\ldots,v_g\}$, on the other hand, are all type II points; for every $i=1, \dots, g$, the corresponding residue field $\widetilde{\mathscr{H}(v_i)}$ of the completed residue field $\mathscr{H}(v_i)$ has transcendence degree one over $\mb{C}$.
 Let $C_{v_i}$ denote the (unique) smooth projective algebraic curve over $\mathbb{C}$ whose field of rational functions $K(C_{v_i})$ equals $\widetilde{\mathscr{H}(v_i)}$. We will see later that each $C_{v_i}$ is a real algebraic curve of genus 1.
 
 \medskip
 For every $i\,=\,1,\,\ldots,\, g$ let $$N(v_i)\,=\,\{w\in V\::\:v_i\text{ is adjacent to }w\}$$ be
the neighborhood of the vertex $v_i$ inside the graph $G$. We then have a bijection $w\longmapsto p(w)$ between the elements $w\in N(v_i)$ and a subset $\mathcal{A}_i=\{p(w)\}_{w\in N(v_i)}$ of $C_{v_i}(\mathbb{C})$ which will be described explicitly in the next section. 
 
\medskip
Finally, we have a function $\ell\,:\,E\longrightarrow\mathbb{R}_{>0}\cup\{\infty\}$ satisfying $0<\ell(e)<\infty$ whenever $e\in\{e_1,\ldots,e_{g-1}\}$ and $\ell(e)=\infty$ otherwise. We conclude that the object $\text{Sk}(\mathcal{X}^\#_0)$ is a {\it metrized complex of algebraic curves over} $\mathbb{C}$ in the sense of \cite{AB}, described as a tuple
$(G=(V,E),\ell,\{(C_{v_i},\mathcal{A}_i)\}_{i=1,\ldots,g})$. 
 
\medskip
 A key point is that the points of $\text{Supp}(R_\pi)$ are marked points of our metrized complex $\text{Sk}(\mathcal{X}^\#_0)$; moreover, our metrized complex naturally carries a real structure, since $(C_{v_i},\mathcal{A}_i)$ is a marked real elliptic curve for every $i=1,\ldots,g$. For these reasons we will refer to a {\it marked metrized complex of algebraic curves over} $\mathbb{R}$.
 
 \medskip
 Each neighborhood $N(v_i)$ now has two types of points. We will call $p(w)\in \mathcal{A}_i$ a {\it point of attachment} of $C_{v_i}$ if $w\in V$ is a point of type II.
 
 \begin{rem}
 \label{relevant_part_skeleton}
 Consider the very affine curve 
 $\mathcal{X}_{\mathbb{K}}^\circ\subset (\mathbb{K}^*)^2$ defined by
 \begin{equation*}
 \mathcal{X}_{\mathbb{K}}^\circ:=\mathcal{X}_{\mathbb{K}}\cap(\mathbb{K}^*)^2=\mathcal{X}_{\mathbb{K}}\setminus \text{Supp}(R_\pi).
 \end{equation*}
 The marked points $\text{Supp}(R_\pi)$ of $\text{Sk}(\mathcal{X}^\#_0)$, together with their corresponding edges
$$\{e_0,e_\infty, e_{i,j}\::\:i=1,\ldots,g,\:j=1,2\}\, ,$$ emerge naturally from the tropicalization process; they represent the points needed to added in order to compactify 
$\mathcal{X}_{\mathbb{K}}^\circ$.
 
 We now construct a metrized complex $\Gamma=(G=(V^{\pr},E),\ell,\{(Y_{v_i},\mathcal{B}_i)\}_{i=1,\ldots,g})$ representing $\mathcal{X}_{\mathbb{K}}^\circ$ using our marked metrized complex $\text{Sk}(\mathcal{X}^\#_0)$. To do so, we keep the same underlying set $E$ and the same function $\ell$ as before, while setting $V^{\pr}=V\setminus  \text{Supp}(R_\pi)$. The marked curves $(Y_{v_i},\mathcal{B}_i)$ are defined as follows: let $\mathcal{C}_i\subset \mathcal{A}_i$ consist of those points that are not points of attachment of $C_{v_i}$. Then $Y_{v_i}=C_{v_i}\setminus \mathcal{C}_i$ and $\mathcal{B}_i=\mathcal{A}_i\setminus \mathcal{C}_i$. Note that 
 \begin{equation}
 \label{decomposition_of_skeleton}
 \text{Sk}(\mathcal{X}^\#_0)=\Gamma\coprod \text{Supp}(R_\pi)
 \end{equation}
 and we will apply this decomposition in the specialization-based analysis of Section~\ref{Section_Specialization}.
 \end{rem}
 
 \section{Specialization via embedded tropicalization}
 \label{Section_Specialization}
 
In this section we will use Viro's theorem on the convergence of amoebas of affine hypersurfaces 
to obtain a precise description of the marked curves $\{(C_{v_i},\mathcal{A}_i)\}_{i=1,\ldots,g}$ 
and of the specialization of the inflection divisor of the complete linear series 
$|\mathcal{L}(kD)|$ on the hyperelliptic curve 
$\mathcal{X}_{\mathbb{K}}$ to the metrized complex of curves $\text{Sk}(\mathcal{X}^\#_0)$ introduced in the previous section. To do so we will use 
the decomposition \eqref{decomposition_of_skeleton}.
 
 \medskip
 As before, let $(\mathbb{K},v_t)$ denote the non-Archimedean Puiseux field valued by the t-adic norm $v_t$, normalized to satisfy $v_t(t)=e^{-1}$. Let $\mathbb{T}=(\mathbb{R}\cup\{-\infty\},\text{max},+)$ denote the tropical semifield, and set 
 \[
 \text{val}\,:=\,\log\circ v_t:\mathbb{K}\longrightarrow\mathbb{T}\, .
 \]
 
 Let $\text{Trop}\,:\,\mathcal{U}^\#_0\,\longrightarrow\,\mathbb{T}^2$ denote the tropicalization
morphism. Since $\mathbb{K}$ is algebraically closed and non-trivially valued, it follows that 
 \begin{equation}\label{Trop_X_0^ast}
 \text{Trop}(\mathcal{U}^{\#}_0)\,=\,\overline{\text{Val}(\mathcal{U}_{\mathbb{K}}(\mathbb{K}))}
 \end{equation}
 where the bar over the right-hand side denotes Euclidean closure. On the other hand, by Kapranov's
theorem, the right-hand side of \eqref{Trop_X_0^ast} is precisely the tropical curve
$V(\text{Trop}(y^2-f))$ associated to the tropical polynomial
$\text{Trop}(y^2-f)\,=\,\text{max}_{i=1,\ldots,2g+1}\{2y,ix-\nu_i\}$. 
 
 \medskip
 Accordingly we get a map
 \begin{equation}
 \label{mortrop}
 \text{Trop}\,:\,\text{Sk}(\mathcal{U}^\#_0)\,\longrightarrow\, V(\text{Trop}(y^2-f))
 \end{equation}
 that is $n$-to-$1$ along an edge $e$ of $V(\text{Trop}(y^2-f))$ with weight $n$. The
compactification of $V(\text{Trop}(y^2-f))$ inside $\mathbb{TP}^1\times \mathbb{TP}^1$ yields the image of $\text{Trop}(\text{Sk}(\mathcal{X}^\#_0))$; see Figure~\ref{fig:tropicalcurve}.
 
\begin{figure}[!htb]
 \centering
 \includegraphics{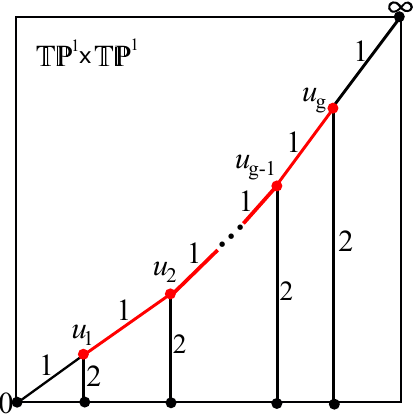}
 \caption{Image of the map $\text{Trop}\,:\,
\text{Sk}(\mathcal{X}^\#_0)\,\longrightarrow\, \mathbb{TP}^1\times\mathbb{TP}^1$. The weight of each
edge represents the local degree of the map.}
\label{fig:tropicalcurve}
\end{figure}
 
 \medskip
Now let $A\,=\,\mathbb{K}[x,y]/(y^2-f)$, so that $\mathcal{U}^\#_0\,=\,\text{an}(\text{Spec}(A),v_t)$. For each $i=1,\ldots,g$, the vertex $v_i\in \text{Sk}(\mathcal{U}^\#_0)$ is sent under the morphism \eqref{mortrop} to the point $u_i\,=\,(a_i,b_i)$ that induces the logarithmic valuation $u_i:A\longrightarrow\mathbb{T}$ defined by
\[
 F(x,y)\,=\,\sum_{m,n}c_{m,n}x^my^n \longmapsto \text{Trop}(F)(u_i)=\text{max}_{m,n}\{\text{val}(c_{m,n})+ma_i+nb_i\}. 
\]

We now compute $\widetilde{\mathscr{H}(u_i)}$. Since $\text{Ker}(u_i)\,=\,(0)$, it follows that
$\mathscr{H}(u_i)$ is the completion of $\text{Frac}(A)$ with respect to $u_i$. In Berkovich's
notation, we have
\[
\widetilde{\mathscr{H}(u_i)}\,=\,\mathscr{H}(u_i)^\circ/\mathscr{H}(u_i)^{\circ\circ}
\]
where $\mathscr{H}(u_i)^\circ\,=\,\{\tfrac{F}{G}\in \text{Frac}(A)\::\:u_i(F)\leq u_i(G)\}$ and $\mathscr{H}(u_i)^{\circ\circ}=\{\tfrac{F}{G}\in \text{Frac}(A)\::\:u_i(F)< u_i(G)\}$.

\medskip
Now suppose that $F(x,y)\,=\,\sum_{m,n}c_{m,n}x^my^n$ is an element of $A$ for which
$\tfrac{F}{1}\,\in\, \mathscr{H}(u_i)^\circ$. Then its residue $\widetilde{F}$ in
$\widetilde{\mathscr{H}(u_i)}$ is the polynomial $\text{in}_{u_i}(F)\,\in\,\mathbb{C}[X,Y]$, where
the parameters $X$ and $Y$ satisfy
\begin{equation}\label{in_f}
Y^2-\sum_{j=2i-1}^{2i+1}a_j X^j \,=\,0.
\end{equation}
Here $\text{in}_{u_i}(F)$ is
the limit of the polynomial $F$ under the flat degeneration defined by the weight $u_i\,\in\,
\mathbb{R}^2$; the equation \eqref{in_f} corresponds to the $u_i$-degeneration of the hyperelliptic equation $y^2-f=0$.

\medskip
It follows that $\widetilde{\mathscr{H}(v_i)}$ is precisely the field $\mathbb{R}(X)[Y]/(\text{in}_{u_i}(y^2-f))$ whenever the polynomial $\sum_{j=2i-1}^{2i+1}a_jX^j$ is separable. Suppose that this is the case, and let $U_i$ be the restriction of real curve $V(Y^2-\sum_{j=2i-1}^{2i+1}a_jX^j)$ to $(\mathbb{C}^*)^2$. In particular, we deduce that $C_{v_i}\setminus\mathcal{A}_i$ and $U_i$ are isomorphic.
 
\medskip
Let $\overline{U_i}$ denote the compactification of $U_i$ inside $\mb{CP}^1 \times \mb{CP}^1$; then $\overline{U_i}$ has geometric genus 1, because the triangle $\Theta_i$ has a single interior lattice point. Note that $\overline{U_i}$ is singular along the boundary; its normalization is precisely $C_{v_i}$. Separability ensures that for $i=1,\ldots,g$ we may write
\begin{equation}
 \label{specialization_of_polys}
 \sum_{j=2i-1}^{2i+1}a_jX^j\,=\, 
 a_{2i+1}X^{2i-1}(X-x_{i,1})(X-x_{i,2})
\end{equation}
 for some $a_{2i+1}$ and $x_{i,j}$ in $\mathbb{C}^*$.
 
 \medskip
 The marked points $\mathcal{A}_i=\{p(w)\}_{w\in N(v_i)}$ of $C_{v_i}$ are in correspondence with the points of $\overline{U_i}\setminus U_i$. The curve $\overline{U_i}$ is singular at $p(w)$ if and only if $w$ corresponds to point of attachment of $C_{v_i}$. 
 
 \begin{rem}
 The upshot of the preceding discussion is that the skeleton $\text{Sk}(\mathcal{X}^\#_0)$ is a combinatorial object which simultaneously contains global information from $\mathcal{X}_{\mathbb{K}}$ and local information from the various elliptic curves $U_i\subset(\mathbb{C}^*)^2$, and it refines the tropical curve $V(\text{Trop}(y^2-f))$. Since it carries all the relevant information of $\mathcal{X}^\#_0$, it may be regarded as the {\bf limit object} of the family $\{\mathcal{X}^\#_\varepsilon\}_{0<\varepsilon \ll 1}$ of real hyperelliptic curves.
 \end{rem}
 
 Our limit construction also works at the level of divisors; we will apply it to the inflection divisor $\text{Inf}(|\mathcal{L}(kD)|)$ associated with $kD=2k\cdot\infty$ over $\mathcal{X}_{\mathbb{K}}$. 
 Accordingly, we define the {\it specialization} map
 \[
\tau_*=\tau_*^{\text{Sk}(\mathcal{X}^\#_0)}:\text{Div}(\mathcal{X}_{\mathbb{K}})\longrightarrow \text{Div}(\text{Sk}(\mathcal{X}^\#_0))
 \]
 to be the composition of the inclusion $\mathcal{X}_{\mathbb{K}}\hookrightarrow \mathcal{X}^\#_0$
 with the retraction $\tau:\mathcal{X}^\#_0\longrightarrow \text{Sk}(\mathcal{X}^\#_0)$. The elements of $\text{Div}(\text{Sk}(\mathcal{X}^\#_0))$ are divisors on metrized complexes of curves {\it that respect the marked points}. Namely, they are of the form $D_{\text{Sk}(\mathcal{X}^\#_0)}\oplus\sum_{i=1}^g D_{v_i}$ where 
 
 \begin{itemize}
 \item $D_{\text{Sk}(\mathcal{X}^\#_0)}$ is a divisor  on the metric graph $(G=(V,E),\ell)$ underlying $\text{Sk}(\mathcal{X}^\#_0)$; 
 \item $D_{v_i}\in\text{Div}(C_{v_i})$ for $i=1,\ldots,g$. Since $C_{v_i}=Y_{v_i}\coprod \mathcal{C}_i$, we require 
 \[
 D_{\text{Sk}(\mathcal{X}^\#_0)}(v_i)=\text{deg}(D_{v_i}|_{Y_{v_i}})
 \]
 and that 
 \[
 D_{\text{Sk}(\mathcal{X}^\#_0)}(v)=D_{v_i}(p(v))
 \]
 for every point $p=p(v)$ in $\mathcal{C}_i$ corresponding to a type I vertex $v\in N(v_i)$. Here $D_{\text{Sk}(\mathcal{X}^\#_0)}(v)$ denotes the coefficient of $D_{\text{Sk}(\mathcal{X}^\#_0)}$ at $v$.
\end{itemize} 
 
 \medskip
 In our particular case, we will show that $\tau_*(\text{Inf}(|\mathcal{L}(kD)|))$ is of the form $D_{\text{Sk}(\mathcal{X}^\#_0)}\oplus\sum_{i=1}^g D_{v_i}$, where 
 \begin{itemize}
 \item $D_{\text{Sk}(\mathcal{X}^\#_0)}$ is a divisor of degree $g(2k-g+1)^2$ on the metric graph $(G=(V,E),\ell)$ underlying $\text{Sk}(\mathcal{X}^\#_0)$ and
 \item $D_{v_i}\in\text{Div}(C_{v_i}\setminus\mathcal{B}_i)$ for $i=1,\ldots,g$,
 subject to 
 \[
 D_{\text{Sk}(\mathcal{X}^\#_0)}(v_i)=\text{deg}(D_{v_i})
 \]
 where $D_{\text{Sk}(\mathcal{X}^\#_0)}(v)$ is the coefficient of $D_{\text{Sk}(\mathcal{X}^\#_0)}$ at $v$.
\end{itemize} 

The fact that $\tau_*(\text{Inf}(|\mathcal{L}(kD)|))$ respects the marked points will follow from Theorems \ref{inflection_along_ell_curves} and \ref{inflection_along_ramification_divisor}.
 
To begin, recall from Proposition \ref{troispartiesdiv} that the divisor $\text{Inf}(|\mathcal{L}(kD)|)$ on $ \mathcal{X}_{\mathbb{K}}$ admits a decomposition
\[
\text{Inf}(|\mathcal{L}(kD)|)= R+S
\]
where $R$ is supported on $R_\pi$ and $S=\text{div}_{\mathcal{X}_{\mathbb{K}}^\circ}(\alpha|_{\mathcal{X}_{\mathbb{K}}^\circ})$, where $\alpha|_{\mathcal{X}_{\mathbb{K}}^\circ}$ is a regular function on $\mathcal{X}_{\mathbb{K}}^\circ=\mathcal{X}_{\mathbb{K}}\setminus R_\pi$ computed as in Equation~\eqref{alpha_in_U}. It follows that
\[
\tau_*(\text{Inf}(|\mathcal{L}(kD)|))=R+\tau_*(\text{div}_{\mathcal{X}_{\mathbb{K}}^\circ}(\alpha|_{\mathcal{X}_{\mathbb{K}}^\circ}));
\]
indeed, $R$ already belongs to the skeleton, so is invariant under the specialization process. It then remains to compute 
\begin{equation}\label{div_w_alpha}
\tau_*(\text{div}_{\mathcal{X}_{\mathbb{K}}^\circ}(\alpha|_{\mathcal{X}_{\mathbb{K}}^\circ})).
\end{equation}
In other words, in order to determine the specialization of the inflection divisor to the skeleton we may ignore all of the marked (i.e., type I) points of $\text{Sk}(\mathcal{X}^\#_0)$ and compute \eqref{div_w_alpha} along the metrized complex $\Gamma$ described in Remark \ref{relevant_part_skeleton}.

\medskip
The specialization \eqref{div_w_alpha} may in fact be realized explicitly as follows. For every $i=1,\ldots,g$, we have an {\it initial coefficient} map
\begin{equation}\label{initial_coeff_map}
\text{ic}:\text{Val}^{-1}(u_i)\longrightarrow U_i
\end{equation}
given by $\text{ic}(\alpha t^{a_i}+\cdots,\beta t^{b_i}+\cdots)=(\alpha,\beta).$ We also have a diagram 

\begin{equation}
\label{initial_coefficient}
 \xymatrix{
 \tau^{-1}(v_i)\ar[r]^{\text{sp}}\ar[d]_{\cong}&C_{v_i}\setminus\mathcal{A}_i\ar[d]^{\cong}\\
 \text{Val}^{-1}(u_i)\ar[r]^-{\text{ic}}&U_i\\
 }
\end{equation}
which is clearly commutative. It follows that the specialization map for divisors on models coincides with the initial coefficient map \eqref{initial_coeff_map} on points, extended by linearity to a map on divisors.

\medskip
Going forward, remember that $C_{v_i}$ refers to the normalization of the singular elliptic curve $\overline{U_i}$ obtained from our $i$-th initial degeneration, as above. Our construction in fact specifies a specialization of {\it linear series} to each curve $C_{v_i}$, $1 \leq i \leq g$, which is both similar to, yet apparently distinct from, limit linear series in the sense of Eisenbud--Harris and Amini--Baker. 

\medskip
Namely, recall that the meromorphic functions $\mathcal{F}=\{\phi_0,\ldots,\phi_{2k-g}\}$ as in \eqref{nafamily} determine a basis for $H^0(\mathcal{L}(kD))$. For each $i$, consider the collection of meromorphic functions
\begin{equation}
    \label{spe-eqs}
    \mathcal{F}(i)\,:=\,\{\widetilde{\phi}_0(i),\ldots,\widetilde{\phi}_{2k-g}(i)\}
\,\subset\, \widetilde{\mathscr{H}(v_i)}
\end{equation}
 canonically induced from $\mathcal{F}$ via the initial degeneration defined by the weight $u_i$. That is,
\[
\widetilde{\phi}_j(i)=X^j\text{ for }0\leq j\leq k\text{, and } \widetilde{\phi}_{j}(i)=X^{j-k-1}Y\text{ for }k+1\leq j\leq 2k-g.
\]
Let $H_i\subset \widetilde{\mathscr{H}(v_i)}$ be the vector space generated by $\mathcal{F}(i)$. This is a real vector space of dimension $2k-g+1$, so it is reasonable to surmise that there exists a $\sigma_{C_{v_i}}$-invariant divisor $D_i$ of degree $2k$ on $C_{v_i}$ such that $H_i\subseteq H^0(\mathcal{L}_{\mathbb{R}}(D_i))$. The following result shows that this is indeed the case.

\begin{lemma}\label{specialization_of_linear_series}
Let $H_i$ be as above, $1 \leq i \leq g$. We have
\[
H_i \subset H^0(C_{v_i},\mc{L}_{\mathbb{R}}(2k \cdot \infty))
\]
where $\infty$ abusively denotes the support of the pullback of $\infty \in \mb{P}^1$ by the (hyperelliptic) structure morphism $C_{v_i} \ra \mb{P}^1$.
\end{lemma}

\begin{proof}
The key to the proof, which follows easily from our construction, is the fact that
\begin{equation}\label{x_and_y}
\begin{split}
\text{div}_{C_{v_i}}(x)&=2\cdot0-2\cdot\infty \text{ and}\\
\text{div}_{C_{v_i}}(y)&=(2i-1)\cdot0+1\cdot(\alpha_1,0)+1\cdot(\alpha_2,0)-(2i+1)\cdot\infty
\end{split}
\end{equation}
for all $i=1,\dots,g$. From the equations~\eqref{x_and_y} we deduce that
$$
\text{div}_{C_{v_i}}(x^j)\,=\, 2j\cdot0-2j\cdot\infty
$$
for all $j\,=\,0, \dots, k$, and
$$
\text{div}_{C_{v_i}}(x^jy)\,=\, (2j+2i-1)\cdot 0+ 1\cdot(\alpha_1,0)+1\cdot(\alpha_2,0)-
(2j+2i-1)\cdot \infty
$$
for all $j\,=\,0,\dots, k-g-1$. The lemma is now clear.
\end{proof}

\medskip
We will now relate the inflection of the limit linear series $(\mc{L}_{\mathbb{R}}(2k \cdot \infty),H_j)$, $j=1,\dots,g$, along the elliptic curves $C_{v_j}$ to the inflection of the original series $|\mc{L}_{\mb{R}}(kD)|$ along the hyperelliptic curve $X$. In the proofs of Theorems~\ref{inflection_along_ell_curves} and \ref{inflection_along_ramification_divisor} below we assume $k \geq g+1$, but a trivial modification of the arguments settles the case $k=g$, with the statements of
the theorems unchanged.

\medskip
More precisely, for $j\,=\,1,\,\ldots,\,g$, begin by compactifying the curve
$V(y^2-\beta x^{2j-1}(x-\alpha_1)(x-\alpha_2))$ inside $\mathbb{P}^1\times\mathbb{P}^1$. Assume
the complex numbers $\be$,
$\al_1$ and $\al_2$ are nonzero, and that $\al_1 \,\neq\, \al_2$. Suppose further that the
polynomial $\beta x^{2j-1}(x-\alpha_1)(x-\alpha_2)$ is real.

\medskip
Let $C_j\,=\,C_{v_j}$ denote the normalization of the curve above; thus $C_j$ is a real elliptic
curve with non-empty real part. The number $n\,=\,n(C_j)$ of components of $C_j(\mb{R})$ is
characterized by the following dichotomy: $n(X_j)\,=\,1$ if and only if $\alpha_1\,=\,
\overline{\alpha_2}$ while $n(C_j)\,=\,2$ if and only if $\alpha_1,\,\alpha_2\,\in \,
\mathbb{R}^*$. The function field $K(C_j)$ is equal to
$\mathbb{R}(x)[y]/(y^2-\beta x^{2j-1}(x-\alpha_1)(x-\alpha_2))$. 

\medskip
We will compute the inflectionary weight of $(\mc{L}_{\mathbb{R}}(2k \cdot \infty),H_j)$ at
each of the four marked points $P \,\in\, \{0,\,(\alpha_1,0),\,(\alpha_2,0),\,\infty\}$ of
$C_j$. Recall that the {\it inflectionary weight} $|P|$ of a linear series $V$ of rank $r$ in a
point $P$ is the total difference between the sequence of vanishing orders of a local basis of
holomorphic sections for $V$ and the generic sequence $(0,\,1,\,\cdots,\,r)$.

\begin{thm}[Inflection in marked points of elliptic curves]\label{inflection_along_ell_curves}
For every $j\,=\,1,\,\cdots,\,g$, let $C_j$ and $H_j \,\subset\, H^0(C_j,\mc{L}_\mathbb{R}(2k
\cdot \infty))$ denote the smooth elliptic curve and linear series, respectively, constructed
above. The inflectionary weights of $H_j$ in the marked points $0, \infty$, $(\al_1,0)$ and
$(\al_2,0)$ are given by
\[
\begin{split}
|0|&= \binom{g+1}{2}+ 2(k-g)(j-1); \\
|\infty|&= \binom{g+1}{2}+ 2(k-g)(g-j); \text{ and } \\
|(\al_1,0)|&=|(\al_2,0)|= \binom{g+1}{2}.
\end{split}
\]
In particular, we have $|0|+|\infty|\,= \,2\binom{g+1}{2}+ 2(k-g)(g-1)$, irrespective of $j$.
\end{thm}

\begin{proof}
We begin by analyzing the case $P\,=\,(\al,0)\,:=\,(\al_j,0)$, where $j=1,2$. Note that
$(x-\alpha)^i\,=\,\sum_{\ell=0}^i\binom{i}{\ell}(-\alpha)^{i-\ell}x^\ell$ belongs to
$H^0(\mathcal{L}_\mathbb{R}(2k \cdot \infty))$ whenever $0 \,\leq\, i \,\leq\, k$ and has
vanishing order $\text{ord}((x-\alpha)^i,(\alpha,0))\,=\,2i$. Similarly, $(x-\al)^iy$ belongs to
$H^0(\mathcal{L}_\mathbb{R}(2k \cdot \infty))$ whenever $0 \,\leq\, i \,\leq \,k-g-1$ and vanishes
to order $2i+1$ in $(\al,0)$. It follows the vanishing sequence of $H_j$ in $(\al,0)$ is
\[
\text{ord}(H_j,(\al,0))\,=\,(0,1,\dots,2(k-g)-1,2(k-g); 2(k-g)+2, 2(k-g)+4,\dots, 2k)
\]
and the inflectionary weight is $|(\al,0)|\,=\,\binom{g+1}{2}$.

\medskip
In a similar vein, the inflection of $H_j$ in $0$ is determined by the vanishing orders of the functions $\mc{F}(j)$ in zero, namely
\[
\text{ord}(x^i,0)= 2i, 0 \leq i \leq k \text{ and }
\text{ord}(x^iy,0) = (2i+2j-1), 0 \leq i \leq k-g-1.
\]
It follows that the inflectionary weight in zero is given by
\[
|0|=2\bigg(\binom{k+1}{2}+ \binom{k-g}{2}\bigg)+ (2j-1)(k-g)- \binom{2k-g+1}{2}= \binom{g+1}{2}+ 2(k-g)(j-1).
\]

Finally, when $P=\infty$, we proceed much as in the $P=0$ case. Indeed, the inflection of $H_i$ in $\infty$ is determined by the pole orders in $\infty$ of the meromorphic functions $\mc{F}(j)$, normalized by the generic pole order $2k$. Namely, we have 
\[
\text{ord}(x^i,\infty)= 2k-2i, 0 \leq i \leq k \text{ and }
\text{ord}(x^iy,\infty) = 2k-(2i+2j+1), 0 \leq i \leq k-g-1.
\]
It follows that
\[
\begin{split}
|\infty|&=-2\binom{k+1}{2}- (2j+1)(k-g)- 2\binom{k-g}{2}- \bigg((-2k)(2k-g+1)+ \binom{2k-g+1}{2}\bigg) \\
&=\binom{g+1}{2}+ 2(k-g)(g-j).
\end{split}
\]
\end{proof}

It is now easy to see that our collection of linear series $\{H_j\}$ satisfies a {\it compatibility} relation in the points of attachment analogous to the defining condition for (refined) Eisenbud--Harris limit linear series \cite{EH}.

\begin{coro}\label{compatibility}
The set of linear series $\{H_j: 1 \leq j \leq g\}$ satisfies the compatibility relation
\[
o_i(H_j,\infty)+ o_{2k-g-i}(H_{j+1},0)= 2k
\]
for all $i=0,\dots,g$ and for all $j=0,\dots,2k-g$, where $o(H_j,P)=(o_1(P),\dots,o_{2k-g}(P))$ denotes the set of vanishing orders of $H_j$ in the point $P$ of the $j$th elliptic component, listed in strictly increasing order. 
\end{coro}

\begin{proof}
This follows immediately from the proof of Theorem~\ref{inflection_along_ell_curves}.
\end{proof}

\begin{rem}\label{unambiguous_specialization} According to our construction, points of the inflection divisor $\text{Inf}(|\mathcal{L}_\mathbb{R}(kD)|)$ of the complete linear series along the hyperelliptic curve specialize unambiguously to particular elliptic components $C_{v_i}$, and never to the edge common to two adjacent type-II vertices of the skeleton or to an infinite length edge. These edges are dual to an edge in the subdivision of the Newton polygon of the hyperelliptic curve linking either the vertices labeled $y^2$ and $x^{2i-1}$ for $i=1,\ldots,g+1$, or the vertices $x^{2i-1}$ and $x^{2i+1}$ for $i=1,\ldots,g$.

Further, any curve of the form $y^2- \nu x^{\mu}=0$ with $\nu \in \mb{C}$ and $\mu \geq 3$ is unramified away from 0 or $\infty$, e.g. because it admits a parametrization by monomials in a single auxiliary variable. 
\end{rem}


Recall the decomposition $\text{Sk}(\mathcal{X}^\#_0)=\Gamma\coprod \text{Supp}(R_\pi)$ from \eqref{decomposition_of_skeleton}, and that $\tau_*(\text{Inf}(|\mathcal{L}(kD)|))=D_{\text{Sk}(\mathcal{X}^\#_0)}\oplus\sum_{i=1}^g D_{v_i}=R+\tau_*(S)$, where $R$ is a divisor supported in $R_\pi$ and $\tau_*(S)=D_\Gamma\oplus\sum_{i=1}^g D_{v_i}^\circ$, where $D_\Gamma$ is supported on $\Gamma$ and  $D_{v_i}^\circ\in\text{Div}(C_{v_i}\setminus \mathcal{C}_i)$. Recall that for $i=1,\ldots,g$,  $\mc{C}_i\subset \mathcal{A}_i$ denotes the collection of marked points on $C_{v_i}$ which are not points of attachment of the curve $C_{v_i}$.

\begin{prop}
Write the divisor Inf$(|\mathcal{L}(kD)|)$ on $\mathcal{X}_{\mathbb{K}}(\mathbb{K})$ in decomposed form as $R+S$, and 
the divisor Inf$(H_i)$ on $C_{v_i}$ correspondingly as $R_i+S_i$, for all $i=1,\dots,g$. We then have
\[
\tau_*(S)=4(k+1)(k-g)\sum_{i=1}^g v_i\oplus_{i=1}^gS_i.
\]
\end{prop}

\begin{proof}
Let $\alpha(i)$ denote the specialization of $\alpha|_{\mathcal{X}_\mathbb{K}^\circ}$ to $K(C_{v_i})$ and let $\alpha|_{U_i}$ denote the Wronskian of the basis \eqref{spe-eqs} restricted to the open set $U_i\subset C_{v_i}$. Then $S=\text{div }\alpha|_{\mathcal{X}_\mathbb{K}^\circ}$ and $S_i=\text{div }\alpha|_{U_i}$. 
The result follows since  $\alpha|_{U_i}=\alpha(i)$.
\end{proof}

Set $\mc{C}=\cup_i \mc{C}_i$. Note that points of ramification divisor $R_{\pi}$ of the hyperelliptic curve specialize to points of $\mc{C}$. Further, given Remark~\ref{unambiguous_specialization}, it is clear that the analogous statement holds at the level of inflection divisors. Namely, let $R_i(\mc{C}_i)$ denote the restriction to $\mc{C}_i$ of the inflection divisor of $H_i$ along $C_{v_i}$, $i=1,\dots,g$. The contribution $R$ of the $\pi$-ramification locus to $\text{Inf}(|\mathcal{L}_\mathbb{R}(kD)|)$ then specializes to the sum of inflectionary loci $\sum_{i=1}^g R_i(\mc{C}_i)$ along the elliptic components supported along $\mc{C}$. The following result implies that the specialization $R \leadsto \sum_{i=1}^g R_i(\mc{C}_i)$ is in fact {\it bijective}.

\begin{thm}[Contribution of $R_{\pi}$ to $\text{Inf}(|\mathcal{L}_\mathbb{R}(kD)|)$]\label{inflection_along_ramification_divisor} Write $\text{Inf}(|\mathcal{L}_\mathbb{R}(kD)|)= R+S$ as in Proposition~\ref{troispartiesdiv}. We have $R=\binom{g+1}{2}R_{\pi}$, where $R_{\pi}$ denotes the ramification divisor of $\pi$. 
\end{thm}

\begin{proof}
Much as in the proof of Theorem~\ref{inflection_along_ell_curves}, we proceed by calculating the inflectionary weight in each point $P \in \mbox{Supp}(R_{\pi})$. If $P \notin \{0,\infty\}$, essentially the same argument used in proving Theorem~\ref{inflection_along_ell_curves} yields $P=\binom{g+1}{2}$. It remains to compute $|0|$ and $|\infty|$. For this purpose, we use the vanishing orders of the basis $\mc{F}$ of $H^0(X,\mc{L}(kD))$, which in turn are prescribed by $\mbox{div}_X(x)$ and $\mbox{div}_X(y)$, much as in Lemma~\ref{specialization_of_linear_series}. This time, we have
\[
\mbox{div}_X(x)= 2 \cdot 0- 2 \cdot \infty \text{ and } \mbox{div}_X(y)= 1 \cdot 0+ R_{\pi}^o- (2g+1) \cdot \infty
\]
where $R_{\pi}^o$ denotes the sum of the $2g$ simple ramification points of $\pi$ that lie inside $\mb{C}^*$. It follows that
\[
\begin{split}
\mbox{ord}(x^i,0)&= 2i, 0 \leq i \leq k \text{ and } \mbox{ord}(x^iy,0)= 2i+1, 0 \leq i \leq k-g-1;\\
\mbox{ord}(x^i,\infty)&= 2k-2i, 0 \leq i \leq k \text{ and } \mbox{ord}(x^iy,\infty)= 2k-2i-2g-1, 0 \leq i \leq k-g-1.
\end{split}
\]
The fact that $|0|\,=\,|\infty|\,=\, \binom{g+1}{2}$ now follows easily.
\end{proof}

The following result is an immediate consequence of Theorem~\ref{inflection_along_ramification_divisor}.
\begin{coro}
Assume that $k>g$. The real linear series $|\mc{L}_{\mb{R}}(kD)|$ then has at least $g(g+1)n(X)$ real inflection points. 
\end{coro}

Finally, the following regeneration-type result sums up how $\text{Inf}(|\mathcal{L}_\mathbb{R}(kD)|)$ compares to the inflection divisors $\mbox{Inf}_{\mb{R}}(H_i)$ associated with the linear series $H_i$ along the elliptic components $C_{v_i}$, $i=1,\dots,g$.
\begin{thm}
\label{convergence}
Fix $g\,\geq\,2$ and $k\,\geq\, g$. Let $f\,=\,\sum_{j=0}^{2g+1}a_jx^j$ be a polynomial of degree
$2g+1$ in $\mathbb{R}[x]$ and let $\nu:[1,2g+1]\cap\mathbb{Z}\longrightarrow \mathbb{Z}_{\geq0}\cup\{\infty\}$ be a function inducing the subdivision $\{[2j-1,2j+1]\::\:j=1,\ldots,g\}$. Suppose that 
 \begin{itemize}
 \item for every $i=1,\ldots,g$, the polynomial $\sum_{j=2i-1}^{2i+1}a_jx^j$ is separable; and
 \item the polynomial $\sum_{j=0}^{2g+1}a_jt^{\nu(j)}x^j$ is separable, and vanishes in $x=0$.
 \end{itemize}
 For every $i$, let $\overline{U_i}$ be the compactification of the curve $V(y^2-\sum_{j=2i-1}^{2i+1}a_jx^j)$ inside $\mb{CP}^1 \times\mb{CP}^1$ and let $C_{v_i}$ be its normalization. Let $H_i$ be the real $g_{2k}^{2k-g}$ on $C_{v_i}$ spanned by the functions $1,x,\ldots,x^g$; $y,xy,\ldots,x^{k-g-1}y$. For 
 $0<\varepsilon \ll 1$, the linear series $|\mathcal{L}_{\mathbb{R}}(kD)|$ on the real hyperelliptic curve $V(y^2-f|_{t=\varepsilon})$ satisfies
\begin{equation}\label{inflection_under_specialization}
\begin{split}
 \text{deg}_{\mathbb{R}}\:\text{Inf}(|\mathcal{L}_{\mathbb{R}}(kD)|)
 &=\sum_{i=1}^g\text{deg}_{\mathbb{R}}\:(\text{Inf}(H_i))- g(g-1)(2k-g+1).
\end{split}
 \end{equation}
\end{thm}

\begin{proof}
Since the polynomial $\sum_{j=0}^{2g+2}a_jt^{\nu(j)}x^j$ is separable, the curve
\[
\mathcal{X}_{\mathbb{K}}(\mathbb{K})\,=\,V(y^2-\sum_{j=0}^{2g+2}a_jt^{\nu(j,0)}x^j)
\]
is hyperelliptic. Since $\sum_{j=2i-1}^{2i+1}a_jx^j$ is separable for
each $i\,=\,1,\,\ldots,\,g$, it follows that the linear series $|\mathcal{L}(kD)|$ on
$\mathcal{X}_{\mathbb{K}}(\mathbb{K})$ specializes to $H_i$ on (the normalization of)
$V(y^2-\sum_{j=2i-1}^{2i+1}a_jx^j)$. It follows that 
\[
\text{deg}\:\text{ Inf}(|\mathcal{L}_\mathbb{R}(kD)|)=\sum_{i=1}^g\text{deg}\:\text{ Inf}(H_i)- (\text{total inflectionary weight in points of attachment}).
\]
On the other hand, Corollary~\ref{compatibility} implies that the inflectionary weight contributed by each of the $(g-1)$ pairs of neighboring points of attachment is computed by
\[
2k(2k-g+1)- 2 \cdot \binom{2k-g}{2}= g(2k-g+1)
\]
which yields \begin{equation}
\begin{split}
 \text{deg} \:\text{Inf}(|\mathcal{L}_{\mathbb{R}}(kD)|)
 &=\sum_{i=1}^g\text{deg} \:(\text{Inf}(H_i))- g(g-1)(2k-g+1).
\end{split}
 \end{equation}.

\medskip
Further, for $0<\varepsilon \ll 1$, the divisor $\text{Inf}(|\mathcal{L}(kD)|)$ on $\mathcal{X}_{\mathbb{K}}(\mathbb{K})$ deforms to the divisor $\text{Inf}(|\mathcal{L}_\mathbb{R}(kD)|)$ of the real hyperelliptic curve $V(y^2-f|_{t=\varepsilon})$, as the projection $q^\#$ is {\it open} above the Archimedean disk $D_\delta$. Finally the deformation $q^\#$ respects the real part, as it is a real deformation.
\end{proof}

It is worth emphasizing that in our construction, our specialization morphism is defined by the initial coefficient morphism \eqref{initial_coefficient}, which in turn is induced by the embedded tropicalization morphism \eqref{mortrop}. We close this section by pointing out a further important difference between our method and that of \cite{AB}. 

\medskip
Let $\Omega$ be the metrized complex of curves obtained from $\text{Sk}(\mathcal{X}_0^\#)$ by
forgetting the marked points $\text{Supp}(R_\pi)$. Namely, $\Omega$ is given by the tuple
$$(G \,=\,(V_\Omega ,\,E_\Omega ),\,\ell_\Omega ,\,\{(C_{v_i},\,\mathcal{G}_i)\}_i)\, ,$$ where
$V_\Omega\,=\,\{v_1,\ldots,v_g\}$ and $E_\Omega\,=\,\{e_1,\ldots,e_{g-1}\}$, while
the function $\ell_\Omega$ is the restriction of $\ell$, and
$\mathcal{G}_i\subset\mathcal{A}_i$ is the set of points of attachment of $C_{v_i}$.

\medskip
Since $V_\Omega$ is a semistable vertex set for $\mathcal{X}_0^\#$, there is a semistable model $\mathfrak{X}$ for $\mathcal{X}_{\mathbb{K}}$ over the valuation ring $\mathbb{K}^{\circ}$ for which the corresponding metrized complex $\mathfrak{CX}$ is the metrized complex $\Omega$.

\medskip
 The Amini--Baker specialization map $$\tau_*^{\mathfrak{CX}}:\text{Div}(\mathcal{X}_{\mathbb{K}})\longrightarrow \text{Div}(\mathfrak{CX})$$
 is constructed as the composition of the reduction morphism $\text{red}:\mathcal{X}_{\mathbb{K}}\longrightarrow\mathfrak{X}_s(\mathbb{C})$ with the retraction morphism $\tau:\mathcal{X}_{\mathbb{K}}\longrightarrow\Gamma$, applying the fundamental identification 
 \[
 \tau^{-1}(v_i)\,=\,\text{red}^{-1}(C_{v_i}\setminus\mathcal{G}_i)\, .
 \]
 
The two specialization maps $\tau_*^{\text{Sk}(\mathcal{X}^\#_0)}$ and $\tau_*^{\mathfrak{CX}}$
differ since in the algebraic setting, the metrized complex $\mathfrak{CX}$ represents an actual
semistable curve $\mathfrak{X}_s(\mathbb{C})$, while $\text{Sk}(\mathcal{X}_0^\#)$ does not.
Finally, note $\mathcal{X}_{\mathbb{C}((t))}$ is defined over the valuation ring
$\mathbb{K}^{\circ}$, so it is equipped with a natural model over $\mathbb{K}^{\circ}$, but
the latter is not semistable. 

\section{Combinatorial construction of curves}\label{Combinatorial_construction}

Fix $k\,>\,g\,>\,1$. In order to construct real algebraic hyperelliptic  curves of genus $g$ such that $|\mathcal{L}_\mathbb{R}(kD)|$ has a controlled number of real inflection points, we apply Theorem~\ref{convergence}. Let $\Theta$ denote the regular subdivision of the triangle $\Delta$ introduced in Section 4. 
The algebraic input for our construction is a $g$-tuple $\{Q_1,\,\ldots,\,Q_g\}$ of polynomials in
$\mathbb{R}[x]$ of the form 
\[
 Q_i= a_ix^{2i-1}(x-x_{i,1})(x-x_{i,2})
\]
for which
\begin{enumerate}
\item $x_{i,1} \neq  x_{i,2}$ and $x_{i,j} \neq 0$ for $i=1, \dots, g$ and $j=1,2$; and
\item the $a_{i}\in\mathbb{R}^*$ satisfy the patchworking conditions $a_{i+1}x_{i,1}x_{i,2}=a_i$ for $i=1,\ldots,g-1$.
\end{enumerate}

Note, in particular, that $\{Q_1,\ldots,Q_g\}$ is uniquely prescribed by the $2g+1$ parameters $\{a_{1},x_{i,j}\:|\:i=1, \dots, g,\:j=1,2\}$. Now let $E_i$ denote the normalization of (the compactification of) $V(y^2-Q_i)\subset \mb{CP}^1 \times \mb{CP}^1$; as explained at the end of Section~\ref{RealHyperCurves}, the fact that $Q_i\in \mathbb{R}[x]$ implies the following dichotomy: 
\begin{itemize}
\item either $x_{i,j}\in\mathbb{C}\setminus\mathbb{R}$, in which case $n(E_i)=1$;
\item otherwise, $n(E_i)=2$.
\end{itemize}

\medskip
The distribution of real inflection points of complete real linear series of degree $d\geq 2$ on a real elliptic curve $E$ was completely characterized in \cite[Thm 3.2.5]{G}; it is predicated on the fact that inflection points are in bijection with $d$-torsion points, which may be visualized on the universal cover of $E$. A natural question is how this result generalizes to a description of real inflection points of an {\it incomplete} real series along $E$. For the sake of completeness, we recall the explicit classification of {\it loc. cit.}, and finish by commenting on how it might be generalized.

\medskip
Accordingly, let $E=(E_\mathbb{C},\sigma_E)$ be a real algebraic curve of genus 1 with
$E(\mathbb{R})\neq\emptyset$, and let $V$ be
a real complete linear series of degree $d\geq 2$. Then $V$ has always $d$ real inflection points when $E(\mathbb{R})$ is connected. On the other hand, when $n(E)=2$, the distribution of real inflection points of $V$ is determined by the parity vector $c(V)\in(\mathbb{Z}/2\mathbb{Z})^2$ according to the following trichotomy:
\begin{enumerate}
\item[(i)] if $c(V)=(1,0)$ or $c(V)=(0,1)$, then
$V$ has $d$ real inflection points located on the connected component of $E(\mathbb{R})$
on which $V$ has odd degree;
 
\item[(ii)] if $c(V)=(0,0)$, then $V$ has $d$ real inflection points on each component;

\item[(iii)] if $c(V)=(1,1)$, then $V$ has no real inflection point.
 \end{enumerate}
 
In our situation, $V$ is not itself complete, but rather embeds in the complete linear series $|\mc{L}|=|\mc{L}(2k \cdot \infty)|$ of even degree $2k$. So because the parity of the series is equal to the sum of the parities of its restrictions to individual real components of $E$, we may disregard possibility (i) when $n(E)=2$ above. Similarly, if $n(E)=2$, then $c(|\mc{L}|)=(0,0)$. So either $E(\mathbb{R})$ is connected or $n(E)\,=\,2$ and $c(|\mc{L}|)\,=\,(0,0)$. We know then that $|\mc{L}|$ has $d$ real inflection points on each component of $E(\mb{R})$, that $|\mc{L}|$ is spanned by sections
\[
\wt{\mc{F}}= \{1,x,\dots,x^k; y,yx, \dots, yx^{k-2}\}
\]
and the question is how these relate, if at all, to the inflection points of the sub-series spanned by
\begin{equation}\label{restricted_basis}
\mc{F}= \{1,x,\dots,x^k; y,yx, \dots, yx^{k-g-1}\}.
\end{equation}

\medskip
Here $y^2=Q(x)$ is an affine equation for $E$, and $Q \in \mb{R}[x]$ is a real cubic polynomial. Note that by fixing for our choice of origin with respect to the group law the point $\infty \in E$ (which we take, as usual, to mean the preimage of $\infty \in \mb{P}^1$ under the projection $(x,y) \longmapsto x$), we obtain a natural bijection between $2k$-torsion points of $E$ and inflection points of $|\mc{L}|$. Explicitly: $P \in E$ is an inflection point of $|\mc{L}|$ if and only if $h^0(\mc{L}(-2kp)) \neq 0$, but $\mc{L}(-2kp)$ is exactly the $2k$-th tensor product $\mc{O}(\infty-p)^{\otimes 2k}$. So because $\infty$ is the origin with respect to the group law, the assertion is clear.

\medskip
Furthermore, since Theorem~\ref{inflection_along_ell_curves}
computes the inflectionary weight of $\mc{F}$ at the boundary points $0, \infty$, $(x_{i,1},0)$ and $(x_{i,2},0)$ of $E_i$, it suffices to count real inflection points of $\mc{F}$ that lie inside $E_i\cap(\mathbb{C}^*)^2$. More precisely, we need to know how many of the $4(k+1)(k-g)$ inflection points of $\mc{F}$ in $E_i\cap(\mathbb{C}^*)^2$ are real. Note that this number is independent of $i$.

\medskip
Note also that Theorem~\ref{inflection_along_ell_curves} yields lower bounds for certain incomplete real
linear series on real elliptic curves. We make this precise as follows.

\medskip
Let $2\leq g\leq k$, and let $H(k)$ denote the real vector space with basis
$$\{x^0,\ldots,x^k;x^0y,\ldots,x^{k-2}y\}\, .$$ Let $V(g) \sub H(k)$ denote the subspace of codimension $g-1>0$ with
basis $$\{x^0,\ldots,x^k;x^0y,\ldots,x^{k-g-1}y\}\, .$$
Then $V(g)$ is a real $g_{2k}^{2k-g}$ on the real elliptic curve $E$. We write $\text{Inf}(V(g))=R+S$ and $s_\mathbb{R}(n(E))=\text{deg}_\mathbb{R}(S)$. We now exhibit real hyperelliptic curves with controlled number of inflection points.


\begin{thm}\label{lower_bounds_for_reality}
Given $k>g>1$ and $1\leq n\leq g+1$, let $a=0$ if $n=g+1$ and let $a=1$ otherwise. There exists a real hyperelliptic curve $X$ of topological type $(g,n,a)$ such that 
 \begin{equation}\label{lower_bounds}
\text{deg}_\mathbb{R}\:\text{Inf}(|\mathcal{L}_\mathbb{R}(kD)|)= (g(g+1)+s_\mathbb{R}(2))(n)+(g-n+1)s_\mathbb{R}(1)-s_\mathbb{R}(2)
\end{equation}
for each pair $(s_\mathbb{R}(1),s_\mathbb{R}(2))$. 
\end{thm}


\begin{proof}
For a given pair  $(s_\mathbb{R}(1),s_\mathbb{R}(2))$, we choose elliptic curves $E_1$ and $E_2$ such that the $S$ part of the inflection divisor of $V(g)$ on $E_i$ has real degree $s_\mathbb{R}(i)$ for $i=1,2$.

\medskip
Then using  $E_1$ a total of $g+1-n$ times and $E_2$ a total of $n-1$ times, we can deform them to obtain a real hyperelliptic curve with $g(g+1)n +(n-1)s_\mathbb{R}(2)+(g-n+1)s_\mathbb{R}(1)$ real inflection points, 
according to Theorem~\ref{inflection_along_ramification_divisor} and \ref{convergence}.
\end{proof}

\medskip
In order to improve upon the amounts in Theorem~\ref{lower_bounds_for_reality}, the key issue is to understand how inflection points of the restricted basis \eqref{restricted_basis} 
are distributed.
 Once this (together with the appropriate generalization of \cite[Thm 3.2.5]{G}) has been achieved, select polynomials $Q_1,\ldots,Q_g\in\mathbb{R}[x]$ as above and set $E_i:=V(y^2-Q_i)$. Let $\nu:\Delta\cap\mathbb{Z}^2\longrightarrow \mathbb{Z}\cup\{\infty\}$ be a function that induces $\Theta$. We then patchwork the polynomials $Q_i$ to form the larger polynomial $Q=y^2-\sum_{i=0}^{2g+2}a_{j,i}t^{\nu(i,0)}x^i$. For $0< \ep \ll 1$, the number of real inflection points of the real hyperelliptic curve defined by $Q|_{t=\ep}$ is then dictated by Theorem~\ref{convergence}. Our result \cite[Thm 5.1]{CG} conjecturally settles the case in which each $E=E_i$ is {\it maximally real}, i.e. has a real locus $E(\mb{R})$ with two connected components.

\end{document}